\documentclass{amsart}

\usepackage{amsmath}
\usepackage{amsthm,amssymb,color,comment}

\usepackage{bbm}
\usepackage{hyperref}
\usepackage[TS1,T1]{fontenc}
\usepackage{dsfont}
\usepackage{tikz}
\usepackage{enumitem}
\usepackage[sort]{cite}% to arrange citations in increasing order

\usepackage{graphics}
\usepackage{mathrsfs}
\usepackage{lmodern}%
\usepackage[english]{babel}%
\usepackage{amsfonts}%
\usepackage{color}
\usepackage{floatrow}
\usepackage[ruled,vlined]{algorithm2e}
\usepackage{subcaption}
\usepackage{cleveref}
\usepackage{geometry}
\usepackage{enumitem}

\numberwithin{equation}{section}

\newtheorem{thm}{Theorem}[section]

\gdef\thref#1{Th.~}

\gdef\tabref#1{Table~\ref{#1}}

\newcommand{\diff}{\mathop{}\!\mathrm{d}}

\DeclareMathOperator{\minT}{min}

\newcommand{\R}{\mathbb{R}}
\newcommand{\N}{\mathbb{N}}

\newcommand{\supp}{\text{supp}}

\makeatletter
\newcommand{\doublewidetilde}[1]{{%
  \mathpalette\double@widetilde{#1}%
}}
\newcommand{\double@widetilde}[2]{%
  \sbox\z@{$\m@th#1\widetilde{#2}$}%
  \ht\z@=.9\ht\z@
  \widetilde{\box\z@}%
}
\makeatother
\definecolor{kuba}{rgb}{0.858, 0.188, 0.478}
\definecolor{minor}{rgb}{0.0, 0.0, 1.0}
\newtheorem{theorem}{Theorem}[section]

\newtheorem{lemma}[theorem]{Lemma}

\newtheorem{definition}[theorem]{Definition}

\theoremstyle{definition}
\newtheorem{assumption}[theorem]{Assumption}
\newtheorem{remark}[theorem]{Remark}
\newtheorem{notation}[theorem]{Notation}

\gdef\thref#1{Th.~}

\gdef\tabref#1{Table~\ref{#1}}

\author{Piotr Gwiazda}
\address{{\it Piotr Gwiazda:} Institute of Mathematics, Polish Academy of Sciences, ul. \'Sniadeckich 8, 00-656 Warsaw, Poland}
\email{pgwiazda@mimuw.edu.pl}
\author{B\l{}a\.zej Miasojedow}
\address{{\it B\l{}a\.zej Miasojedow:} Faculty of Mathematics, Informatics and Mechanics, University of Warsaw, ul. Banacha 2, 02-097 Warsaw, Poland}
\email{bmiasojedow@mimuw.edu.pl}
\author{Jakub Skrzeczkowski}
\address{{\it Jakub Skrzeczkowski:} Faculty of Mathematics, Informatics and Mechanics, University of Warsaw, ul. Banacha 2, 02-097 Warsaw, Poland}
\email{jakub.skrzeczkowski@student.uw.edu.pl}
\author{Zuzanna~Szyma\'nska}
\address{{\it Zuzanna Szyma\'nska:} ICM, University of Warsaw, ul. Tyniecka 15/17, 02-630 Warsaw, Poland, and Institute of Mathematics, Polish Academy of Sciences, ul. \'Sniadeckich 8, 00-656 Warsaw, Poland}
\email{mysz@icm.edu.pl}

\thanks{%P.~Gwiazda, 
B.~Miasojedow, J.~Skrzeczkowski and Z.~Szyma\'nska acknowledge the support from the National Science Centre Poland Grant 2017/26/M/ST1/00783. The calculations were made with the support of the Interdisciplinary Center for Mathematical and Computational Modelling (ICM) of the University of Warsaw under the computational grant no. G79-28.}

\begin{document}

\title[EBT for model with discontinuous interaction kernel]{Convergence of the EBT method for a non-local model of cell proliferation with discontinuous interaction kernel}

\begin{abstract}
We consider the EBT algorithm (a particle method) for the non-local equation with a discontinuous interaction kernel. The main difficulty lies in the low regularity of the kernel which is not Lipschitz continuous, thus preventing the application of standard arguments. Therefore, we use the radial symmetry of the problem instead and transform it using spherical coordinates. The resulting equation has a Lipschitz kernel with only one singularity at zero. We introduce a new weighted flat norm and prove that the particle method converges in this norm. We also comment on the two-dimensional case which requires the application of the theory of measure spaces on general metric spaces and present numerical simulations confirming the theoretical results. In a companion paper, we apply the Bayesian method to fit parameters to this model and study its theoretical properties.
\end{abstract}

\keywords{particle method, EBT algorithm, measure solutions, flat metric, non-local equation, convergence analysis, cancer modelling}

\maketitle
\tableofcontents

%%%%%%%%%%%%%%%%%%%%%%%%%%%%%%%%%%%%%%%%%%%%%%%%%%%%%%%%%%%%%%%%%%%%%%%%%%%
%%%%%%%%%%%%%%%%%%%%%%%%%%%%%%%%%%%%%%%%%%%%%%%%%%%%%%%%%%%%%%%%%%%%%%%%%%%
%%%%%%%%%%%%%%%%%%%%%%%%%%%%%%%%%%%%%%%%%%%%%%%%%%%%%%%%%%%%%%%%%%%%%%%%%%%

\section{Introduction}\label{Sec:Introduction:Math}

%%%%%%%%%%%%%%%%%%%%%%%%%%%%%%%%%%%%%%%%%%%%%%%%%%%%%%%%%%%%%%%%%%%%%%%%%%%
%%%%%%%%%%%%%%%%%%%%%%%%%%%%%%%%%%%%%%%%%%%%%%%%%%%%%%%%%%%%%%%%%%%%%%%%%%%
%%%%%%%%%%%%%%%%%%%%%%%%%%%%%%%%%%%%%%%%%%%%%%%%%%%%%%%%%%%%%%%%%%%%%%%%%%%

\noindent In this paper we study a numerical algorithm to solve the non-local equation 
\begin{equation}\label{non-local_proliferation}
  \partial _t n(x,t) \ =\  k* n(x,t) \,\Bigl( 1 \ -\  n(x,t) \Bigr),
\end{equation}
where $k = \mathds{1}_{B_{\sigma}(0)}/|B_{\sigma}(0)|$ is so-called interaction kernel, $B_{\sigma}(0)$ denotes a ball centred at 0 with radius $\sigma$ and volume $|B_{\sigma}(0)|$ while
$$
k \ast n(x,t) \ =\ \int_{\R^d} k(x-y) \, n(y,t)\, \diff y.
$$
Recently, we proposed the above model to describe cells' proliferation within a solid tumour \cite{szymanska2021} where we also used Bayesian methodology to estimate model parameters. Although our modelling approach alters from those proposed earlier, there are examples of other interesting models describing the dynamics of multicellular spheroids like the ones proposed by Byrne and Chaplain \cite{byrne1995,byrne1996,byrne1998} studied in a bunch of analytical papers \cite{cui2003,chen2005,friedman2006a,friedman2006b}.\\

\noindent In the case $k$ is compactly supported and Lipschitz continuous, one can solve \eqref{non-local_proliferation} using classical particle method \cite{MR3986559,deRoos1988,MR3267354,MR3591128}, originally studied in the context of fluid dynamics and kinetic theory \cite{MR2647756, MR802214, MR1040146, MR987390, MR731212, MR3529992, MR2888301, MR3632260} and brought to mathematical biology by de Roos \cite{de1997gentle,de2001physiologically,deRoos1988}. This can be possibly combined with the splitting technique \cite{MR3870087,carrillo2014,MR2050900}. In these algorithms, one divides the population into smaller groups called cohorts. This allows transforming PDE \eqref{non-local_proliferation} to a system of ODEs for the masses of each cohort (localisations of cohorts are constant as there is no transport term in \eqref{non-local_proliferation}).\\

\noindent To prove numerical convergence of these methods, one embeds the problem into the space of non-negative Radon measures \cite{carrillo2012,MR2997595,MR2644146,MR2746205,MR3342408,MR3507552,MR3461738} where each cohort is represented by a Dirac mass. The convergence above is shown with respect to the flat norm (bounded Lipschitz distance) and its main properties are reviewed in Section \ref{sect:measure_theory}. One can also prove the convergence with respect to weak$^*$ topology on the space of measures \cite{MR3138105, MR2888301} but this is formally a weaker result and it does not yield convergence estimate as weak$^*$ topology is not explicitly metrizable. Moreover, a simple variant of flat norm proved to be useful for optimal control problems in spaces of measures  \cite{MR4045015, MR4066016, MR4027078} which may result in the future application of particle methods for such problems.  \\

\noindent We remark that particle methods have been used in the more general context of structured population models (with transport term) and we refer to \cite[Chapter 4]{our_book_ACPJ} for the systematic treatment of this topic. We also note that these methods are widely used in the community of theoretical biologists and ecologists, see for instance \cite{FalsterE2719,Falster2016,zhang2017performance}.\\

\noindent In this paper we deal with the case $k = \mathds{1}_{B_{\sigma}(0)}/|B_{\sigma}(0)|$ so that $k$ has jump at the boundary of $B_{\sigma}(0)$ and is not even continuous. However, one may use radial symmetry of the problem to gain some regularity: it turns out that after moving to the spherical coordinates, \eqref{non-local_proliferation} becomes of the form
\begin{equation}\label{eq:transform_coordinates_INTRODUCTION}
\partial _t p(R,t) \  =\  \left(4\pi R^2 \ -\  p(R,t) \right) \int_0^\infty L(R,r)\, p(r,t) \diff r
\end{equation}
where $p(R,t) = 4\pi R^2\, n((0,0,R),t)$ and $L(R,r)$ is given by \eqref{eq:defL}. It turns out that after appropriate modification of the flat norm, one can apply the particle method. More precisely, we write
\begin{equation}\label{eq:approximation_INTRODUCTION}
p(\cdot, t) \approx \sum_{i=1}^N m_i(t) \, \delta_{x_i}(\cdot), \qquad x_i = \frac{i}{N} \, R_0,
\end{equation}
where $\delta_{x_i}$ denotes Dirac measure at $x_i$ representing particle, $m_i$ is the mass concentrated at $x_i$ while $R_0$ is some parameter restricting the domain of interest (in general, the equation enjoys infinite-speed-of-propagation property and the solution is not compactly supported). Inserting \eqref{eq:approximation_INTRODUCTION} into \eqref{eq:transform_coordinates_INTRODUCTION} yields formally system of ODEs for masses
\begin{equation}\label{eq:ODE_num_scheme_INTRODUCTION}
\partial_t m_i(t) = \Big( 4\pi x_i^2 \, \frac{R_0}{N} -\  m_i(t) \Big) \sum_{j=1}^N L(x_i,x_j)\, m_j(t) 
\end{equation}
where $m_i(0)$ are chosen so that $p(\cdot, 0) \approx \sum_{i=1}^N m_i(0) \, \delta_{x_i}(\cdot)$. Equation \eqref{eq:ODE_num_scheme_INTRODUCTION} is solvable by some standard algorithms, for instance Euler or Runge-Kutta method. \\

\noindent The main result of this paper reads:
\begin{theorem}\label{thm:main_res_convergence}
Let $p(r,0) = 4\pi r^2 \, n_0(r)$ where $n_0: \R^+ \to \R^+$ is bounded and compactly supported. Consider approximation of $p(r,0)$ in the space of measures
$$
\mu_0^N(\cdot) = \sum_{i=1}^N m_i(0) \, \delta_{x_i}(\cdot), \qquad m_i(0) = \int_{x_{i-1}}^{x_{i}} p(r,0) \diff r.
$$
Let $p(r,t)$ be the solution to \eqref{non-local_proliferation} with initial condition $p(r,0)$ and $\mu_{t}^N = \sum_{i=1}^N m_i(t) \, \delta_{x_i}$ where $m_i(t)$ solve \eqref{eq:ODE_num_scheme_INTRODUCTION}. Then, there is a constant $C$ independent of $R_0 > 1$ and $N$ such that
\begin{equation}\label{eq:error_in_main_res}
\left\|p(\cdot,t) -  \mu_{t}^N \right\|_{BL^*,w} \leq C\,\frac{R_0^2}{N} +  C\, e^{-R_0}
\end{equation}
where the weighted norm $\|\cdot \|_{BL^*,w}$ is defined in \eqref{defeq:flatnorm_weighted} and $p(\cdot,t)$ is identified with the measure $p(\cdot,t)(A) = \int_{A} p(r,t) \diff r$.
\end{theorem}

\noindent We remark that the term $e^{-R_0}$ in the error estimate in Theorem \ref{thm:main_res_convergence} comes from the fact that the solution is supported on the whole line $\R^+$, even if initial data is compactly supported. In other words, this term represents error coming from the truncation of the support of the solution.\\

\noindent The first novelty of this paper concerns application of radial symmetry of the problem to gain sufficient regularity of the kernel. After the change of variables in Section \ref{sect:radial_change}, using weighted norm introduced in \eqref{defeq:flatnorm_weighted}, we somehow incorporate singularity at $R,r =0$ into the definition of the norm. A crucial observation is the following inequality
$$
\left|\partial_{R} L(R,r)\right| \leq \frac{1}{R} \left(  \frac{2\sigma}{r} + {L(R,r)} \right),
$$
which measures in sufficiently optimal way singularity of Lipschitz constant of $L$ as $R, r \to 0$. The factors $\frac{1}{R}, \frac{1}{r}$ will be incorporated into the definition of the weighted flat norm cf. \eqref{defeq:flatnorm_weighted}. 
\\

\noindent Another novelty of this paper is related to the case $d=2$. It is known that particle method convergence is related to the Lipschitz regularity of $L$ but for $d=2$ we only know that (except $R,r = 0$) $L$ is only $1/2$-H\"older continuous cf. \eqref{eq:defL_2d}. However, one may observe that if $L$ is only $1/2$-H\"older continuous with respect to the usual Euclidean metric, it is Lipschitz continuous with respect to the H\"older metric $d_{1/2} = |x-y|^{1/2}$. This results in a different order of convergence with respect as discussed in estimates in Section \ref{sect:2Dcase}. This is based on the recent monograph \cite{our_book_ACPJ} presenting the theory of measure spaces on general metric spaces. \\

\noindent The structure of the paper is as follows. In Section \ref{sect:measure_theory} we review necessary concepts from measure theory including weighted flat norm. Then, in Section \ref{sect:radial_change} we perform the radial change of variables and we study properties of the radial kernel $L(R,r)$.  Section \ref{sect:radial_sln_measure_equation} is devoted to the well-posedness of the radial equation \eqref{eq:transform_coordinates_INTRODUCTION}, including continuity estimates. In Section \ref{sect:infinite_speed_of_prop} we obtain estimates for \eqref{eq:transform_coordinates_INTRODUCTION} that allow us to neglect the effect of infinite speed of propagation so we can assume the support of the solution to be bounded. Moreover, we show how to interpret solutions to the numerical scheme as measure solutions so that they can be compared with exact measure solutions. Finally, in Section \ref{sect:main_convergence} we prove the main convergence result. Section \ref{sect:2Dcase} discusses necessary changes to handle two dimensional case while Section \ref{sect:sim_res} is devoted to presentation of numerical simulations confirming theoretical results.

%%%%%%%%%%%%%%%%%%%%%%%%%%%%%%%%%%%%%%%%%%%%%%%%%%%%%%%%%%%%%%%%%%%%%%%%%%%
%%%%%%%%%%%%%%%%%%%%%%%%%%%%%%%%%%%%%%%%%%%%%%%%%%%%%%%%%%%%%%%%%%%%%%%%%%%
%%%%%%%%%%%%%%%%%%%%%%%%%%%%%%%%%%%%%%%%%%%%%%%%%%%%%%%%%%%%%%%%%%%%%%%%%%%

\section{Relevant measure theory}\label{sect:measure_theory}

%%%%%%%%%%%%%%%%%%%%%%%%%%%%%%%%%%%%%%%%%%%%%%%%%%%%%%%%%%%%%%%%%%%%%%%%%%%
%%%%%%%%%%%%%%%%%%%%%%%%%%%%%%%%%%%%%%%%%%%%%%%%%%%%%%%%%%%%%%%%%%%%%%%%%%%
%%%%%%%%%%%%%%%%%%%%%%%%%%%%%%%%%%%%%%%%%%%%%%%%%%%%%%%%%%%%%%%%%%%%%%%%%%%

\noindent Let $\mathcal{M}(\R^+)$ be the space of bounded real-valued signed Borel measures on $S$ cf. \cite[Sections 1.3, 3.1]{Folland.1984}. Intuitively, if $\mu \in \mathcal{M}(\R^+)$ then $\mu$ assigns a real number to each measurable subset $A \subset S$ which is a measure of the subset $\R^+$. This generalises distributions with densities in the sense that if $\mu$ has density $n(x)$, we have
$$
\mu(A) = \int_{A} n(x) \diff x.
$$
We note that the space of measures is a vector space; in particular, the difference between two measures is again a measure. We also write $\mathcal{M}^+(\R^+)$ for the subset of non-negative measures on $\R^+$ i.e. when $\mu \in \mathcal{M}^+(\R^+)$, we have $\mu(A) \geq 0$ for all subsets $A \subset \R^+$. From the point of view of applications, non-negative measures model biological quantities like size, age, or spread of population.\\

\noindent {\bf Hahn-Jordan decomposition.} We recall that if $\mu \in \mathcal{M}^+(\R^+)$ is the signed measure, there are (uniquely determined) two non-negative measures $\mu^+, \mu^- \in \mathcal{M}^+(\R^+)$ with disjoint supports such that
$$
\mu = \mu^+ - \mu^-.
$$
We call $(\mu^+, \mu^-)$ the Hahn-Jordan decomposition of $\mu$.\\

\noindent {\bf Norms on the space of measures.} %To perform analysis in spaces of measures, one needs to equip them with norms. There are three meaningful choices that will be exploited below: total variation, flat norm and weighted flat norm. 
To perform analysis in spaces of measures, one needs to equip them with norms. Three meaningful choices will be exploited below: total variation, flat norm, and weighted flat norm.\\

\noindent {\bf Total variation.} If $\mu \in \mathcal{M}(\R^+)$, we define total variation of $\mu$ as
\begin{align}\label{eq:TVnorm}
		\|\mu\|_{TV}:= \mu^+(\R^+) + \mu^{-}(\R^+)
\end{align}
where $\mu^+$, $\mu^-$ is the Hahn-Jordan decomposition of $\mu$. \\

\noindent {\bf Flat norm.} The second is the flat norm (or bounded Lipschitz distance) defined as
\begin{align}\label{defeq:flatnorm}
		\|\mu\|_{BL^*}:= \sup\left\{\int_{\R^+} \! \psi  \,\mathrm{d}\mu : \psi \in BL(\R^+), \|\psi\|_{BL} \leq 1\right\},
\end{align} 
where the space of bounded Lipschitz functions $BL(\R^+)$ is given by
\begin{equation}\label{eq:def_BL}
	BL(\R^+)=\left\{f:\R^+ \to \R \mbox{ is continuous and } \|f\|_{\infty}<\infty, |f|_{Lip}<\infty\right\},
\end{equation}
 where  
 \begin{equation}\label{app:basic_norms}
 \|f\|_{\infty}=\underset{x\in \R^+}{\sup}\,|f(x)|, \qquad \qquad |f|_{Lip}=\underset { x\neq y}{\sup}\, \frac {|f(x)-f(y)|}{|x-y|}.
 \end{equation}
Space $BL(\R^+)$ is equipped with the norm 
\begin{equation}\label{eq:def_BLnorm}
\|f\|_{BL} = \max\left(\|f\|_{\infty}, \, |f|_{Lip}\right) \leq \|f\|_{\infty} + |f|_{{Lip}}.
\end{equation}
Bounded Lipschitz distance in the space of measures has been used frequently in recent years, for instance, to study structured population models \cite{carrillo2012}, numerical algorithms \cite{MR3986559,carrillo2014} or segregation in cross-diffusion systems \cite{MR3870087}.\\

\noindent Now, we list some simple properties that are helpful when one works in the flat norm setting. For the proof see Remark 1.23, Proposition 1.44 and Theorem C.2 in \cite{our_book_ACPJ} as well as \cite[Theorems 4, 6; Section~5.8]{MR2597943}.
\begin{lemma}\label{lem:prod_BL_bel}
Let $f, g \in BL(\R^+)$ and $\mu \in \mathcal{M}(\R^+)$.
\begin{itemize}
    \item[(A)] We have $\left| \int_{\R^+} f(x)\, \diff \mu(x) \right| \leq \|f\|_{BL} \, \|\mu\|_{BL^*}$.
    \item[(A')] If $h \in L^{\infty}(\R^+)$ only then $\left| \int_{\R^+} h(x)\, \diff \mu(x) \right| \leq \|h\|_{\infty} \, \|\mu\|_{TV}$.
    \item[(B)] If additionally $\mu \in \mathcal{M}^+(\R^+)$ then $\|\mu\|_{TV} = \|\mu\|_{BL^*}$.
    \item[(C)] We have $f\,g \in BL(\R^+)$ and $\|f \, g\|_{BL} \leq 2 \, \|f\|_{BL} \, \|g\|_{BL}$.
    \item[(D)] (Rademacher's theorem) $f \in BL(\R^+)$ if and only if $f \in W^{1,\infty}(\R^+)$ i.e. $f, f' \in L^{\infty}(\R^+)$. Moreover, 
    $$
    |f |_{Lip} \leq \|f'\|_{\infty}, \qquad 
    $$
\end{itemize}
\end{lemma}
\begin{proof}[Sketch of the proof]
For (A) we note that $f/\|f\|_{BL}$ is bounded by 1 in $BL(\R^+)$ so that by \eqref{defeq:flatnorm} $\left| \int_{\R^+} \frac{f(x)}{\|f\|_{BL}} \diff \mu(x) \right| \leq \|\mu\|_{BL^*}$. For (B) we note that we always have $\|\mu\|_{TV} \geq \|\mu\|_{BL^*}$ and the opposite inequality follows by choosing $\psi(r) = 1$ in \eqref{defeq:flatnorm} (this uses $\mu \in \mathcal{M}^+(\R^+)$). For (C) we observe that $\|f\, g\|_{\infty} \leq \|f\|_{\infty} \, \|g\|_{\infty} \leq \|f\|_{BL} \, \|g\|_{BL}$. Moreover, for $x, y \in \R^+$ we have
$$
|f(x)\, g(x) - f(y)\,g(y)| \leq \left(\|f\|_{\infty} \, |g|_{Lip} + \|g\|_{\infty} \, |f|_{Lip}\right) |x-y| \leq 2 \|f\|_{BL} \, \|g\|_{BL} \, |x-y|
$$
so that $|f \, g|_{Lip} \leq 2 \|f\|_{BL} \, \|g\|_{BL}$. Hence, (C) follows. For (D) see \cite[Theorems 4, 6; Section~5.8]{MR2597943}.
\end{proof}

\noindent {\bf Weighted flat norm.} This is the analog of the flat norm applicable for singular problems like in this paper. Given non-negative function $f: \R^+ \to \R^+$ and $\mu \in \mathcal{M}(\R^+)$ we define measure $\mu/f$ with
\begin{equation}\label{eq:def_weighted_measure}
\frac{\mu}{f}(A) = \int_{A} \frac{1}{f(r)} \diff \mu(r).
\end{equation}
Whenever this measure is bounded (i.e. $\left\|\frac{\mu}{f} \right\|_{TV} < \infty$) it makes sense to write $\left\| \frac{\mu}{f}  \right\|_{BL^*}$ which is the weighted flat norm of $\mu$ with weight $f$. The most important case is $f(r) =r$. We define
\begin{equation}\label{def:weighted_measure_space}
\mathcal{M}_{w}(\R^+) := \left\{ \mu \in \mathcal{M}(\R^+): \left\| \frac{\mu}{r} \right\|_{TV} < \infty \right\}
\end{equation}
We also write $\mathcal{M}^+_{w}(\R^+)$ for the subspace of non-negative measures in $\mathcal{M}_{w}(\R^+)$. The idea is that $\mathcal{M}_{w}(\R^+)$ consists of measures that vanish at least linearly at $r = 0$. For $\mu \in \mathcal{M}_{w}(\R^+)$ we define 
\begin{align}\label{defeq:flatnorm_weighted}
		\|\mu\|_{BL^*,w}:= \sup\left\{\int_{\R^+} \! \frac{\psi(r)}{r}  \,\mathrm{d}\mu(r) : \psi \in BL(\R^+), \|\psi\|_{BL} \leq 1\right\},
\end{align}
We conclude with one of the most useful properties of flat norm.
\begin{lemma}
Space $(\mathcal{M}_{w}^+(\R^+), \| \cdot \|_{BL^*,w})$ is a complete metric space.
\end{lemma}

\begin{proof}
Let $\{\mu_n\}_{n \in \N}$ be a Cauchy sequence in $(\mathcal{M}_{w}^+(\R^+), \| \cdot \|_{BL^*,w})$. Let $\nu_n = \mu_n/r$. Then, $\{\nu_n\}_{n \in \N}$ is a Cauchy sequence in $(\mathcal{M}(\R^+), \| \cdot \|_{BL^*})$. As $(\mathcal{M}^+(\R^+), \| \cdot \|_{BL^*})$ is a complete metric space \cite[Theorem 1.61]{our_book_ACPJ} or \cite[Theorem 2.7 (ii)]{MR2644146}, $\nu_n \to \nu$ in $(\mathcal{M}^+(\R^+), \| \cdot \|_{BL^*})$. Let $\mu = r \, \nu$. Then, $\mu \in \mathcal{M}_{w}^+(\R^+)$. We claim that $\mu_n \to \mu$ in $(\mathcal{M}_{w}^+(\R^+), \| \cdot \|_{BL^*,w})$. Indeed, for all $\psi \in BL(\R^+)$ with $\| \psi \|_{BL} \leq 1$ we have
$$
\int_{\R^+} \frac{\psi(r)}{r} \diff (\mu_n - \mu)(r) = \int_{\R^+} \psi(r) \diff (\nu_n - \nu)(r) \leq \| \nu_n - \nu \|_{BL^*}. 
$$
Taking supremum over left-hand side
$$
\| \mu_n - \mu\|_{BL^*,w} \leq \| \nu_n - \nu \|_{BL^*} \to 0.
$$
\end{proof}
\noindent The most important property of flat norm is that in topology generated by flat norm, any measure can be approximated with an appropriate combination of Dirac masses. 
\begin{lemma}\label{lem:approx_distr}
Let $\mu \in \mathcal{M}^+[0,R_0]$ and $f: \R^+ \to \R^+$ be such that $\left\|\frac{\mu}{f} \right\|_{TV} = \int_{[0,R_0]} \frac{1}{f(r)}\diff \mu(r) < \infty$. Consider $\mu^N = \sum_{k=1}^N \mu(A_k^N) \, \delta_{x_k^N}$ with $x_k^N = \frac{k}{N}R_0$ and 
$$
A_k^N = \left[\frac{k-1}{N}R_0, \frac{k}{N}R_0 \right) \subset \R \mbox{ for } k=1,...,N-1 \quad \mbox{ and } \quad A_N^N = \left[\frac{N-1}{N}R_0, R_0 \right] \subset \R.
$$
Then,
$$
\left\|\frac{\mu^N - \mu}{f} \right\|_{BL^*[0,R_0]} \leq \frac{R_0}{N} \, \left\|\frac{\mu}{f} \right\|_{TV}.
$$
\end{lemma}
\begin{proof}
Directly from the definition we obtain
$$
\left\|\frac{\mu^N - \mu}{f} \right\|_{BL^*[0,R_0]} = \sup_{\psi} \int_{[0,R_0]} \frac{\psi(r)}{f(r)} \diff(\mu^N - \mu)(r) = \sup_{\psi}
\sum_{k=1}^N \int_{A_k^N} \frac{\left(\psi(x_k^N) -\psi(r)\right)}{f(r)} \diff \mu(r),
$$
where the sumpremum is taken above all $\psi \in BL[0,R_0]$ with $\|\psi\|_{BL}\leq 1$. As $\left|\psi(x_k^N) -\psi(r)\right| \leq \frac{R_0}{N}$, the proof is concluded.
\end{proof}
\begin{remark}\label{cor_approximation_lebesgue_measure} 
The same proof as in Lemma \ref{lem:approx_distr} shows the following. Consider the usual Lebesgue measure $\lambda$ on $[0,R_0]$. Let $\lambda^N = \frac{R_0}{N} \sum_{k=1}^N \delta_{x_k^N}$ with $x_k^N = \frac{k}{N}R_0$. Then,
$$
\left\|\lambda^N - \lambda \right\|_{BL^*[0,R_0]} \leq \frac{R_0^{2}}{N}.
$$
\end{remark}

%%%%%%%%%%%%%%%%%%%%%%%%%%%%%%%%%%%%%%%%%%%%%%%%%%%%%%%%%%%%%%%%%%%%%%%%%%%
%%%%%%%%%%%%%%%%%%%%%%%%%%%%%%%%%%%%%%%%%%%%%%%%%%%%%%%%%%%%%%%%%%%%%%%%%%%
%%%%%%%%%%%%%%%%%%%%%%%%%%%%%%%%%%%%%%%%%%%%%%%%%%%%%%%%%%%%%%%%%%%%%%%%%%%

\section{Radial change of variables}\label{sect:radial_change}

%%%%%%%%%%%%%%%%%%%%%%%%%%%%%%%%%%%%%%%%%%%%%%%%%%%%%%%%%%%%%%%%%%%%%%%%%%%
%%%%%%%%%%%%%%%%%%%%%%%%%%%%%%%%%%%%%%%%%%%%%%%%%%%%%%%%%%%%%%%%%%%%%%%%%%%
%%%%%%%%%%%%%%%%%%%%%%%%%%%%%%%%%%%%%%%%%%%%%%%%%%%%%%%%%%%%%%%%%%%%%%%%%%%

\noindent In this section, we transform the original problem \eqref{non-local_proliferation} using radial change of variables. Then, we study the basic properties of the resulting radial interaction kernel that are relevant for algorithm convergence.
\begin{thm}[Radial change of coordinates]\label{thm1:radial_change}
Let $n(x,t)$ be the solution to \eqref{non-local_proliferation} with radially symmetric initial condition $n_0(x)$. Then the radial density $p(R,t)$ defined with
$$
    p(R,t) = 4\pi R^2\, n((0,0,R),t), \qquad \qquad p_0(R) = 4\pi R^2\, n_0((0,0,R)).
$$
satisfies equation
\begin{equation}\label{non-local_proliferation_polar_3D}
\partial _t p(R,t) \  =\  \left( 4\pi R^2 \ -\  p(R,t) \right) \,  \int_0^\infty L(R,r)\, p(r,t)  \diff r
\end{equation}
where the radial interaction kernel $L$ is given by
\begin{equation}\label{eq:defL}
L(R,r) =\frac{3}{16\,\pi\,\sigma^{3}} \, \frac{\min\{(R+r)^2,\sigma^2\}-\min\{(R-r)^2,\sigma^2\}}{R \, r}
\end{equation}
\end{thm}
\begin{proof}
Let $K:[0,\infty) \to [0,\infty)$ be defined with $K(|x|) = k(x)$ and $L:[0,\infty)\to \R$ be defined with 
\begin{equation}\label{eq:def_fun_L}
L'(r^2)=K(r),
\end{equation}
which can be computed explicitly
\begin{equation}\label{l}
    L(r)=\frac{3}{4\,\pi\,\sigma^{3}} \, \min\{r,\sigma^2\}.
\end{equation}
\noindent Since the initial condition is a radially symmetric function and equation \eqref{non-local_proliferation} does not involve space derivatives, the solution is a radially symmetric function. We let $p(R, t) = {4\pi R^2} \, n ((0,0,R), t)$, where $R = | x | = (x_1^2 + x_2^2 + x_3^2)^{ 1/2}$. We fix such $x \in \R^3$ and $R \geq 0$. We make two simple observations. First, points $(x_1,x_2,x_3)$ and $(0,0,R)$ are in the same distance from zero. Second, the convolution $k*n$ is also a radially symmetrical function. Therefore we have
\begin{equation}\label{splot1}
\begin{split}
&k*n(x,t) = k*n((0,0,R),t) = \\ 
&\qquad =\int_{\R^3}K\left(\left((0-y_1)^2+(0-y_2)^2+(R-y_3)^2\right)^{1/2}\right)n((y_1,y_2,y_3),t) \diff y \\
&\qquad =\int_{\R^3}K\left(\left(y_1^2+y_2^2+y_3^2+R^2-2\,R\,y_3\right)^{1/2}\right)\frac{p\left(\left(y_1^2+y_2^2+y_3^2\right)^{1/2},t\right)}{4\pi \left(\left(y_1^2+y_2^2+y_3^2\right)^{1/2}\right)^2}\diff y.
\end{split}
\end{equation}

\noindent To convert \eqref{non-local_proliferation} to polar coordinates we substitute
\begin{equation}\label{eq:change_of_var_R3}
y_1 = r \cos{\alpha}\cos{\beta}, \quad
y_2 = r \sin{\alpha}\cos{\beta}, \quad 
y_3 = r \sin{\beta},  
\end{equation}
where $r>0$, $0 \leq \alpha \leq 2\pi$ and $-\frac{\pi}{2} \leq \beta \leq \frac{\pi}{2}$. The Jacobian determinant of the change of variables in \eqref{eq:change_of_var_R3} is equal to $r^2 \cos \beta$. Using
$$
r^2=y_1^2+y_2^2+y_3^2, \qquad \qquad 2\,R\,y_3=2\,R\,r\sin\beta
$$
to \eqref{splot1}, we get the following 
\begin{align*}
& \int_0^\infty \int_0^{2\pi}\int_{-\pi/2}^{\pi/2}  K\left((r^2+R^2-2\,R\,r\sin\beta)^{1/2}\right) \, {\frac{p(r,t)}{4\pi r^2}} \, r^2 \, \cos\beta \diff \beta \diff \alpha \diff r\\
& \qquad \qquad = \frac{1}{2} \int_0^\infty\int_{-\pi/2}^{\pi/2} K\left((r^2+R^2-2\,R\,r\sin\beta)^{1/2}\right)\, p(r,t)\,\cos\beta \diff\beta \diff r \\
& \qquad \qquad =\frac{1}{4R} \int_0^\infty \int_{(R-r)^2}^{(R+r)^2} K(u^{1/2})\,p(r,t)\,\frac{1}{r} \diff u \diff r,
\end{align*}

\noindent where the last equality comes from the substitution $u=r^2+R^2-2\,R\,r\sin \beta$. Now, integrating with respect to $u$ and using function $L$ from \eqref{eq:def_fun_L} we obtain 
\begin{equation*}
  K*p(R,t)=\frac{1}{4}\int_0^\infty \frac{L((R+r)^2)-L((R-r)^2)}{R \, r}\,p(r,t) \diff r, 
\end{equation*}
so that using \eqref{l} we deduce
\begin{equation*}
    K*p(R,t)= {\frac{3}{16\,\pi\,\sigma^{3}}}\int_0^\infty \frac{\min\{(R+r)^2,\sigma^2\}-\min\{(R-r)^2,\sigma^2\}}{R \, r} \, p(r,t) \diff r.
\end{equation*}

\noindent All together we obtain \eqref{non-local_proliferation_polar_3D}.
\end{proof}

\noindent In what follows it will be useful to introduce
\begin{equation}\label{eq:defLtilde}
\widetilde{L}(R,r):= \min\{(R+r)^2,\sigma^2\}-\min\{(R-r)^2,\sigma^2\}
\end{equation}
so that $L(R,r) = \frac{\widetilde{L}(R,r)}{R\,r}$. We conclude with a lemma collecting useful properties of $\widetilde{L}(R,r)$ and $L(R,r)$.
\begin{lemma}\label{lem:prop_of_L}
Let $L(R,r)$ and $\widetilde{L}(R,r)$ be given with \eqref{eq:defL} and \eqref{eq:defLtilde} respectively. Then,
\begin{enumerate}[label=(P\arabic*)] \setlength\itemsep{0.25em}
    \item \label{propA} $0 \leq L(R,r) \leq 4$,
    \item \label{propB} $\widetilde{L}(R,r) \leq 2\sigma^2$, $|\partial_R \widetilde{L}(R,r)| \leq 4\sigma$ and $\widetilde{L}(R,r) \in BL(\R^+)$ with norm $2\sigma^2 + 4\sigma$,
    \item \label{propC} $\left|\partial_{R} L(R,r)\right| \leq \frac{1}{R} \left(  \frac{2\sigma}{r} + {L(R,r)} \right)$,
    \item \label{propD} $L(R,r)$ is supported only on the set $|R-r|\leq \sigma$,
    \item \label{propE} $\frac{\widetilde{L}(R,r)}{R} \leq 8 \sigma$,
    \item \label{propF} $r^2\, L(R,r) \leq 4 \sigma^2$,
    \item \label{propG} $\left|\int_{\R^+} L(R,r) \, \psi(R) \diff R\right| \leq 8 \sigma \|\psi\|_{\infty}$ for all bounded $\psi$,
    \item \label{propH} $\left|\int_{\R^+} \frac{\widetilde{L}(R,r)}{R} \, \psi(R) \diff R\right| \leq 16 \sigma^2 \|\psi\|_{\infty}$ for all bounded $\psi$,
    \item \label{propI} $r \mapsto \int_{\R^+} \psi(R) \, {\widetilde{L}(R,r)} \diff R \in BL(\R^+)$ with norm $8\sigma^3 \|\psi \|_{\infty} + 8\sigma^2 \|\psi \|_{\infty}$.
\end{enumerate}
\end{lemma}
\begin{proof}
First, we prove \ref{propA}, \ref{propB} and \ref{propC} simultaneously distinguishing four cases.
\begin{itemize}
\item[(1)] If $|R+r|\leq \sigma$ and $|R-r|\leq \sigma$ we have $L(R,r) = 4$ and $\widetilde{L}(R,r)=4Rr$ so that \ref{propA} is satisfied. For \ref{propB} we have $\widetilde{L}(R,r) \leq 2\sigma^2$, $\partial_R \widetilde{L}(R,r) = 4r \leq 4\sigma$ and $\partial_r \widetilde{L}(R,r) = 4R \leq 4\sigma$ because $0\leq r, R \leq \sigma$. Finally, \ref{propC} is clear because $\partial_{R} L(R,r) = 0$.
\item[(2)] If $|R+r|\geq \sigma$ and $|R-r|\geq \sigma$ we have $L(R,r)=0$ and the conclusion is obvious.
\item[(3)] If $|R+r|\geq \sigma$ and $|R-r|\leq \sigma$ we have $L(R,r) = \frac{\sigma^2 - (R-r)^2}{Rr}$. As $\sigma^2 \leq (R+r)^2$ we have
$$
L(R,r) = \frac{\sigma^2 - (R-r)^2}{Rr} \leq \frac{(R+r)^2 - (R-r)^2}{Rr} = 4
$$
so \ref{propA} is proved. To prove \ref{propB} we note that $\widetilde{L}(R,r)=\sigma^2 - (R-r)^2 \leq \sigma^2$. Then, 
$$
\partial_R \widetilde{L}(R,r) = - \partial_r \widetilde{L}(R,r) = -(R-r) 
$$
which is bounded by $\sigma$. For \ref{propC} we observe that
$$
\partial_R L(R,r) = \frac{-(R-r)Rr - (\sigma^2 - (R-r)^2)\,r}{R^2\,r^2} = - \frac{R-r}{Rr} - \frac{\sigma^2 - (R-r)^2}{R^2r}.
$$
For the first term we have $\left| \frac{R-r}{Rr} \right| \leq \frac{\sigma}{Rr}$ while the second is simply $-\frac{L(R,r)}{R}$.
\item[(4)] As $|R-r| \leq R+r = |R+r|$ there is no case $|R+r| < \sigma$ and $|R-r| > \sigma$.
\end{itemize}
Then, \ref{propD} follows from the second case while \ref{propE} is a consequence of \ref{propA} and \ref{propD}. Indeed, when $r \leq 2\sigma$ we can estimate $\frac{\widetilde{L}(R,r)}{R} = r \,L(R,r) \leq 8 \sigma$. On the other hand, when $r \geq 2\sigma$, we can assume $R \geq \sigma$ so that $\frac{\widetilde{L}(R,r)}{R} \leq \frac{2\sigma^2}{\sigma} = 2\sigma$. To prove \ref{propF} we proceed in the similar manner. If $r \leq 2\sigma$ we can estimate $r^2 \, L(R,r) \leq 4 \sigma^2$. Otherwise, $R \geq r-\sigma$ so 
$$
r^2 \, L(R,r) \leq \frac{r}{r - \sigma} \, \widetilde{L}(R,r) = \frac{1}{1 - \sigma/r}  \, \widetilde{L}(R,r) \leq \frac{1}{1 - 1/2}  \, \widetilde{L}(R,r) \leq 4\sigma^2.
$$

\noindent To see \ref{propG} we compute using \ref{propD} and \ref{propA}:
$$
\left|\int_{\R^+} L(R,r) \, \psi(R) \diff R\right| = \left|\int_{\R^+} L(R,r) \, \mathds{1}_{|R-r|\leq \sigma} \, \psi(R) \diff R\right| \leq 4 \|\psi \|_{\infty} \int_{\R^+}  \mathds{1}_{|R-r|\leq \sigma} \diff R = 8\sigma \|\psi \|_{\infty}
$$
as the measure of the set $|R-r|\leq \sigma$ equals $2\sigma$. Assertion \ref{propH} is proved similarly, using \ref{propE} instead of \ref{propA}. Finally, to see \ref{propI} we first observe that the map of interest is bounded because thanks to \ref{propB} and \ref{propD}
$$
\left|\int_{\R^+} \widetilde{L}(R,r) \, \psi(R) \diff R\right| = \left|\int_{\R^+} \widetilde{L}(R,r) \, \mathds{1}_{|R-r|\leq \sigma} \, \psi(R) \diff R\right| \leq 4\sigma^2 \|\psi \|_{\infty} \int_{\R^+}  \mathds{1}_{|R-r|\leq \sigma} \diff R = 8\sigma^3 \|\psi \|_{\infty}.
$$
Similarly, using \ref{propB} we prove that $\left|\int_{\R^+} \partial_r\widetilde{L}(R,r) \, \psi(R) \diff R\right| \leq 8\sigma^2 \|\psi \|_{\infty}$ and this concludes the proof of \ref{propI} thanks to Radamacher's Theorem cf. Lemma \ref{lem:prod_BL_bel} (D).
\end{proof}

%%%%%%%%%%%%%%%%%%%%%%%%%%%%%%%%%%%%%%%%%%%%%%%%%%%%%%%%%%%%%%%%%%%%%%%%%%%
%%%%%%%%%%%%%%%%%%%%%%%%%%%%%%%%%%%%%%%%%%%%%%%%%%%%%%%%%%%%%%%%%%%%%%%%%%%
%%%%%%%%%%%%%%%%%%%%%%%%%%%%%%%%%%%%%%%%%%%%%%%%%%%%%%%%%%%%%%%%%%%%%%%%%%%

\section{Measure solutions to radial equation}\label{sect:radial_sln_measure_equation}

%%%%%%%%%%%%%%%%%%%%%%%%%%%%%%%%%%%%%%%%%%%%%%%%%%%%%%%%%%%%%%%%%%%%%%%%%%%
%%%%%%%%%%%%%%%%%%%%%%%%%%%%%%%%%%%%%%%%%%%%%%%%%%%%%%%%%%%%%%%%%%%%%%%%%%%
%%%%%%%%%%%%%%%%%%%%%%%%%%%%%%%%%%%%%%%%%%%%%%%%%%%%%%%%%%%%%%%%%%%%%%%%%%%

\noindent The solution to the numerical scheme is represented as a generalised solution being a measure rather than just a function. This section is devoted to the formulation of radial equation \eqref{eq:transform_coordinates_INTRODUCTION} in the space of measures.\\

\noindent We introduce
$$
E_T = C(0,T; (\mathcal{M}^+_{w}(\R^+), \| \cdot \|_{BL^*,\,w})),
$$
to be the space of continuous curves indexed with time $t \in [0,T]$ and valued in $(\mathcal{M}^+_{w}(\R^+), \| \cdot \|_{BL^*,\,w})$. Here, $\mathcal{M}^+_{w}(\R^+)$ is the subspace of non-negative measures vanishing linearly at $r = 0$, cf. \eqref{def:weighted_measure_space} and $\| \cdot \|_{BL^*,w}$ is defined with \eqref{defeq:flatnorm_weighted}. Space $E_T$ is equipped with usual supremum norm
\begin{equation}\label{eq:norm_on_E}
\| \mu_{\bullet}\|_{E_T} = \sup_{t \in [0,T]} \| \mu_t \|_{BL^*,\,w}.
\end{equation}
Moreover, we write $\lambda$ for the Lebesgue measure on $\R^+$ and let $S(R) = 4\pi R^2$.
\begin{definition}[mild measure solution]\label{def:measure_solution}
We say that $ \mu_{\bullet} =
\left\{\mu_t\right\}_{t\in[0,T]} \in E_T $
is a mild measure solution to \eqref{eq:transform_coordinates_INTRODUCTION} with initial condition $\mu_0 \in \mathcal{M}^+_w(\R^+)$ if for all test functions $\psi \in BL(\R^+)$ we have
\begin{equation}\label{eq:duh_solution_concept_TF}
\begin{split}
&\int_{\R^+} \frac{\psi(R)}{R} \diff\mu_t(R) = \int_{\R^+} \frac{\psi(R)}{R} \, e^{-\int_0^t \int_{\R^+} {L}(R, r) \diff \mu_s(r) \diff s }\diff\mu_0(R) \,+\\ &\qquad \qquad + \int_0^t \int_{\R^+} \int_{\R^+} \frac{\psi(R)}{R}\,  e^{-\int_0^{t-s} \int_{\R^+} {L}(R,r) \diff \mu_u(r) \diff u } \,   L(R,r) \, S(R) \diff \mu_s(r) \diff \lambda(R) \diff s .
\end{split}
\end{equation}
\end{definition}
\begin{notation}
It is convenient to define for $\mu_{\bullet} \in E_T$, $R \in \R^+$ and $t \in [0,T]$ 
\begin{equation}\label{not:simplificationE}
\mathcal{E}[\mu_{\bullet}, R, t] := e^{-\int_0^t \int_{\R^+} \mathcal{L}(R,r) \diff \mu_s(r) \diff s}.
\end{equation}
\end{notation}
\noindent {\bf Motivation for Definition \ref{def:measure_solution}:} Suppose that $\mu_t$ is a measure given by a continuous density $p(r,t)$ so that $\mu_t(A) = \int_{A} p(r,t) \diff \mu_t(r)$. As \eqref{eq:duh_solution_concept_TF} holds for all $\psi \in BL(\R^+)$ we discover
$$
\frac{p(R,t)}{R} = e^{-\int_0^t \int_{\R^+} {L}(R, r) \, p(r,s) \diff s } \, \frac{p(R,0)}{R} + 
\int_0^t \int_{\R^+} \frac{1}{R}\,  e^{-\int_0^{t-s} \int_{\R^+} {L}(R,r) \, p(r,u) \diff u } \, L(R,r) \, S(R) \, p(r,s) \diff r \diff s.
$$
We multiply with $R > 0$ to get
$$
{p(R,t)} = e^{-\int_0^t \int_{\R^+} {L}(R, r) \, p(r,s) \diff s } \, {p(R,0)} + 
\int_0^t \int_{\R^+} e^{-\int_0^{t-s} \int_{\R^+} {L}(R,r) \, p(r,u) \diff u } \, L(R,r) \, S(R) \, p(r,s) \diff r \diff s.
$$
This is precisely variation-of-constants formula (or Duhamel formula) for \eqref{eq:transform_coordinates_INTRODUCTION}. Thus Definition \ref{def:measure_solution} provides a generalisation of the concept of classical solutions to \eqref{eq:transform_coordinates_INTRODUCTION}. The factor $\frac{1}{R}$ is designed to study solutions that decay at least linearly when $R \to 0$. \\

\noindent The main result of this section reads:
\begin{theorem}\label{thm:well-posed_flat_metric_radial}
Suppose that $\mu_0$ is such that $\left\|\frac{\mu_0}{r^2}\right\|_{TV}, \left\|\frac{\mu_0}{r}\right\|_{TV}, \left\|{\mu_0}\right\|_{TV} < \infty$. Then, there exists a unique non-negative mild measure solution to \eqref{eq:transform_coordinates_INTRODUCTION} in $E_T$ with initial condition $\mu_0$. It satisfies the bound 
\begin{equation}\label{eq:unif_bound_solutions}
\|\mu_t\|_{BL^*,w} \leq \|\mu_0\|_{BL^*,w} \, e^{C(\sigma)t}, \qquad \qquad \|\mu_t\|_{TV} \leq \|\mu_0\|_{TV} \, e^{C(\sigma)t},
\end{equation}
for some constant $C(\sigma)$ depending only on $\sigma$. Moreover, the solution is Lipschitz continuous in time
\begin{equation}\label{eq:continuity_in_time_measure_solution}
   \|\mu_{t_2} - \mu_{t_1} \|_{BL^*,w} \leq C_d \, |t_2 - t_1|.
\end{equation}
Furthermore, if $\mu_t^{(1)}$, $\mu_t^{(2)}$ are measure solutions with initial conditions $\mu_0^{(1)}$, $\mu_0^{(2)}$ satisfying assumptions above, we have
\begin{equation}\label{eq:important_continuity_estimate} 
\begin{split}
\left\| \mu_t^{(1)} - \mu_t^{(2)} \right\|_{BL^*,w}  \leq  C_d \, \left\|\frac{\mu_0^{(1)} - \mu_0^{(2)}}{r\, \minT(1,r)}\right\|_{BL^*}.
\end{split}
\end{equation}
The constant $C_d$ depends on data and $\left\|\frac{\mu_0}{r^2}\right\|_{TV}, \|\mu_0\|_{BL^*,w}, \left\|\frac{\nu_0}{r^2}\right\|_{TV}, \|\nu_0\|_{BL^*,w}$.
\end{theorem}
\begin{remark}\label{rem:ass_initial_condition_theorem}
Note that we have in mind the situation when $\mu_0$ is a measure with density $p(r) = 4\pi r^2 \, n_0(0,0,r)$ where $n_0 \in L^{\infty}(\R^3)$ and has compact support cf. Theorem \ref{thm1:radial_change} so that $\mu_0$, $\frac{\mu_0}{r}$ and $\frac{\mu_0}{r^2}$ are bounded in total variation norm.
\end{remark}
\noindent The proof of Theorem \ref{thm:well-posed_flat_metric_radial} relies on continuity properties of (RHS) of \eqref{eq:duh_solution_concept_TF}. 
To this end, given $\mu_{\bullet}, \nu_{\bullet} \in E_T$, two measures $\gamma_1, \gamma_2 \in \mathcal{M}^+_w(\R^+)$ we define two families of measures with
\begin{equation}\label{eq:duh_solution_concept_tildemu}
\begin{split}
&\int_{\R^+} \frac{\psi(R)}{R} \diff \widetilde{\mu_t}(R) = \int_{\R^+} \frac{\psi(R)}{R}\, \mathcal{E}[\mu_{\bullet}, R, t] \, \diff \gamma_1(R)\, + \\
& \qquad \qquad \int_0^t  \int_{\R^+} \int_{\R^+} \frac{\psi(R)}{R}\, \mathcal{E}[\mu_{\bullet}, R, t-s] \,  {L}(R,r) \, S(R)\diff \mu_s(r) \diff R \diff s =: X_1 + Y_1,
\end{split}
\end{equation} 
\begin{equation}\label{eq:duh_solution_concept_tildenu}
\begin{split}
&\int_{\R^+} \frac{\psi(R)}{R} \diff \widetilde{\nu_t}(R) = \int_{\R^+} \frac{\psi(R)}{R}\, \mathcal{E}[\nu_{\bullet}, R, t] \, \diff \gamma_2(R)\, + \\
& \qquad \qquad  \int_0^t \int_{\R^+} \int_{\R^+} \frac{\psi(R)}{R}\,  \mathcal{E}[\nu_{\bullet}, R, t-s] \,   {L}(R,r) \, S(R) \diff \nu_s(r) \diff R \diff s =: X_2 + Y_2.
\end{split}
\end{equation}

\begin{notation}[Constant $C_d$]{\label{notation:Cdatac}}
For the sake of simplicity, we introduce constant $C_d$ which may differ from line to line and is allowed to depend continuously on data, i.e.
$$
L,\,   \sigma, \, T,\, \left\|\frac{\gamma_1}{r^2} \right\|_{TV},\, \left\|\frac{\gamma_2}{r^2} \right\|_{TV},\, \|\mu_{\bullet}\|_{E_T},\, \|\nu_{\bullet}\|_{E_T}.
$$
\end{notation}

\begin{lemma}[continuity of $\mathcal{E}$]\label{lem:E_basic_estimates}
Under notation above, $\mathcal{E}$ enjoys the following properties:
\begin{enumerate}[label=(E\arabic*)]
    \item \label{E1}boundedness: $0 \leq \mathcal{E}[\mu_{\bullet}, R, t] \leq 1$,
    \item \label{E2}continuity with respect to measure argument:
    \begin{align*}
    |\mathcal{E}[\mu_{\bullet}, R, t] - \mathcal{E}[\nu_{\bullet}, R, t]| \leq \, \frac{C(\sigma)}{R} \, \int_0^t \| \mu_s - \nu_s \|_{BL^*,w} \diff s
    \end{align*}
    \item \label{E3} continuity with respect to time: 
    $$
    \sup_{R \in \R^+} \left|\mathcal{E}[\mu_{\bullet}, R, t] - \mathcal{E}[\mu_{\bullet}, R, s]\right| \leq \frac{C_d}{R} \,|t-s|, 
    $$
    \item \label{E4} continuity with respect to $R$: the map $\R \ni R \mapsto \minT(R,1)\, \mathcal{E}[\mu_{\bullet}, R, t]$ is in $BL(\R^+)$ with norm bounded by $C(\sigma, \|\mu_{\bullet}\|_{E_T})$.
\end{enumerate}
\end{lemma}

\begin{proof}
Assertion \ref{E1} is obvious. For \ref{E2} we compute:
\begin{multline*}
 \left|e^{-\int_0^t \int_{\R^+} {L}(R,r) \diff \mu_s(r) \diff s} - e^{-\int_0^t \int_{\R^+} {L}(R,r) \diff \nu_s(r) \diff s} \right| 
 \leq \left| \int_0^t \int_{\R^+} {L}(R,r) \diff(\mu_s - \nu_s)(r) \diff s \right|  \\
 \leq \frac{1}{R}\left| \int_0^t \int_{\R^+} \frac{\widetilde{L}(R,r)}{r} \diff(\mu_s - \nu_s)(r) \diff s \right| \leq \frac{4\sigma^2 + 4\sigma}{R} \int_0^t \|\mu_s - \nu_s\|_{BL^*,w} \diff s,
\end{multline*}
where we used that the function $e^{x}$ is 1-Lipschitz for $x \leq 0$ and that $\widetilde{L}(R,r)$ is jointly in $BL(\R^+)$ with constant $4\sigma^2 + 4\sigma$ cf. Lemma \ref{lem:prop_of_L} \ref{propB}. To verify \ref{E3}, we compute similarly
\begin{multline*}
\left|e^{-\int_0^t \int_{\R^+} {L}(R,r) \diff \mu_u(r) \diff u}  - e^{-\int_0^s \int_{\R^+} {L}(R,r) \diff \mu_u(r) \diff u}  \right| \leq \int_s^t \int_{\R^+} \frac{\widetilde{L}(R,r)}{R\,r} \diff \mu_u(r) \diff u \leq \\
\leq \frac{4\sigma^2 + 4\sigma}{R} \, \|\mu_{\bullet}\|_{E_T} \, |t-s|.
\end{multline*}
To see \ref{E4}, we first observe that the map is bounded thanks to \ref{E1}. To prove Lipschitz continuity, we use Radamacher's theorem cf. Lemma \ref{lem:prod_BL_bel} (D). We observe that thanks to the product rule, the derivative of this map with respect to $R$ equals
\begin{align*}
\mathds{1}_{R \leq 1} \, e^{-\int_0^t \int_{\R^+} {L}(R,r) \diff \mu_s(r) \diff s} + \mbox{min}(R,1)\, e^{-\int_0^t \int_{\R^+} {L}(R,r) \diff \mu_s(r) \diff s}\, \int_0^t \int_{\R^+} \partial_R{L}(R,r) \diff \mu_s(r) \diff u.
\end{align*}
The first term is clearly bounded with 1 while for the second we estimate using Lemma \ref{lem:prop_of_L} \ref{propC}:
\begin{align*}
&\left|\mbox{min}(R,1)\, e^{-\int_0^t \int_{\R^+} {L}(R,r) \diff \mu_s(r) \diff s}\, \int_0^t \int_{\R^+} \partial_R{L}(R,r) \diff \mu_s(r) \diff s \right| \leq \\
 & \qquad \qquad \qquad \leq \frac{\mbox{min}(R,1)}{R}\, e^{-\int_0^t \int_{\R^+} {L}(R,r) \diff \mu_s(r) \diff s}\, \int_0^t \int_{\R^+} \frac{2\sigma}{r}  \diff \mu_s(r) \diff s \, + \\
 & \qquad \qquad \qquad \phantom{\leq} ~~ + \frac{\mbox{min}(R,1)}{R}\, e^{-\int_0^t \int_{\R^+} {L}(R,r) \diff \mu_s(r) \diff s}\, \int_0^t \int_{\R^+} {L(R,r)} \diff \mu_s(r) \diff s := X + Y.
\end{align*}
Term $X$ is bounded because $\frac{\mbox{min}(R,1)}{R} \leq 1$ and $\int_{\R^+} \frac{2\sigma}{r}  \diff \mu_s(r) \leq 2\sigma \|\mu_{\bullet} \|_{E_T}$. Term $Y$ is bounded because the function $x \, e^{-x}$ is bounded for $x \geq 0$.
\end{proof}
\begin{lemma}[a priori estimate]\label{lem:aprioriestimateinweightedspace}
Let $\widetilde{\mu_t}$ be defined with \eqref{eq:duh_solution_concept_tildemu}. Then, there is a constant $C(\sigma)$ depending only on $\sigma$ such that
$$
 \left\| \widetilde{\mu_t}\right\|_{BL^*,w} \leq  \left\| \gamma_1 \right\|_{BL^*,w} +  C(\sigma) \int_0^t \left\| {\mu_s}\right\|_{BL^*,w} \diff s.
$$
Similarly,
$$
 \left\| \widetilde{\mu_t}\right\|_{TV} \leq  \left\| \gamma_1 \right\|_{TV} +  C(\sigma) \int_0^t \left\| {\mu_s}\right\|_{TV} \diff s.
$$
\end{lemma}
\begin{proof}
We use \eqref{eq:duh_solution_concept_tildemu} which we rewrite here for convenience:
\begin{equation}\label{eq:semigroup_formulation_for_aprioribounds}
    \begin{split}
&\int_{\R^+} \frac{\psi(R)}{R} \diff \widetilde{\mu_t}(R) = \int_{\R^+} \frac{\psi(R)}{R}\, \mathcal{E}[\mu_{\bullet}, R, t] \, \diff \gamma_1(R)\, + \\
& \qquad \qquad \qquad + \int_0^t  \int_{\R^+} \int_{\R^+} \frac{\psi(R)}{R}\, \mathcal{E}[\mu_{\bullet}, R, t-s] \,  {L}(R,r) \, S(R)\diff \mu_s(r) \diff R \diff s.
\end{split}
\end{equation} 
As $\mathcal{E}[\mu_{\bullet}, R, t] \in [0,1]$, the first term is bounded with $\| \psi \|_{\infty} \, \left\| \frac{\gamma_1}{R} \right\|_{TV} = \| \psi \|_{\infty} \, \left\| \gamma_1 \right\|_{BL^*,w}$. In the second term, we write $\frac{\psi(R)}{R}\,{L}(R,r) \, S(R) = 4\pi \, \psi(R) \, \frac{\widetilde{L}(R,r)}{r}$. Thanks to Lemma \ref{lem:prop_of_L} \ref{propI}, the integral $\int_{\R^+} \psi(R) \, \mathcal{E}[\mu_{\bullet}, R, t-s] \, \widetilde{L}(R,r) \diff R \leq \|\psi\|_{\infty} \, C(\sigma)$. Hence,
\begin{multline*}
\int_0^t  \int_{\R^+} \int_{\R^+} \frac{\psi(R)}{R}\, \mathcal{E}[\mu_{\bullet}, R, t-s] \,  {L}(R,r) \, S(R)\diff \mu_s(r) \diff R \diff s \leq \\ \leq  C(\sigma) \int_0^t  \int_{\R^+} \frac{1}{r} \diff \mu_s(r) \diff s = C(\sigma) \int_0^t \left\| {\mu_s}\right\|_{BL^*,w} \diff s.
\end{multline*}
Taking supremum over all $\psi$ such that $\|\psi\|_{BL} \leq 1$ we conclude the proof of the first estimate. For the second, we apply \eqref{eq:semigroup_formulation_for_aprioribounds} with $\psi(R) := \psi_n(R)$ defined as 
$$
\psi_n(R) =
\begin{cases}
 R &\mbox{ if } 0 \leq R \leq n,\\
 {n\, (n+1 - R)} &\mbox{ if } n \leq R \leq n+1, \\ 
 0 &\mbox{ if } R \geq n.
\end{cases}
$$
Note that $\frac{\psi_n(R)}{R} \in [0,1]$ and $\frac{\psi_n(R)}{R} \to 1$ monotonically as $n \to \infty$. As $\gamma_1$ is non-negative, we easily deduce $\left|\int_{\R^+} \frac{\psi_n(R)}{R}\, \mathcal{E}[\mu_{\bullet}, R, t] \, \diff \gamma_1(R)\right| \leq \| \gamma_1 \|_{TV} = \| \gamma_1 \|_{BL^*}$. For the second term we first estimate
\begin{multline*}
\left|\int_0^t  \int_{\R^+} \int_{\R^+} \frac{\psi_n(R)}{R}\, \mathcal{E}[\mu_{\bullet}, R, t-s] \,  {L}(R,r) \, S(R)\diff \mu_s(r) \diff R \diff s\right| \leq \\ \leq \int_0^t  \int_{\R^+} \int_{\R^+}  {L}(R,r) \, S(R)\diff \mu_s(r) \diff R \diff s.
\end{multline*}
Then, we split the integral for two cases: $R \geq 2\sigma$, $R \leq 2 \sigma$ and denote the resulting integrals with $I_1$ and $I_2$ respectively. If $R \geq 2\sigma$, we use the support of $L$ cf. Lemma \ref{lem:prop_of_L} (P4)  to observe that $r \geq R - \sigma \geq \sigma$. Hence, $L(R,r) = \frac{\widetilde{L}(R,r)}{R\,r} \leq \frac{\widetilde{L}(R,r)}{R\,(R-\sigma)}$ and we have
$$
L(R,r) \, S(R) = 4\pi \, \widetilde{L}(R,r) \, \frac{R^2}{R\,(R-\sigma)} \leq 8\pi \, \widetilde{L}(R,r) \mbox{ for } R \geq 2\sigma. 
$$
As $\int_{\R^+} \widetilde{L}(R,r) \diff R \leq C(\sigma)$, we deduce $I_1 \leq C(\sigma) \int_0^t \| \mu_s\|_{TV} \diff s$. The case $R \leq \sigma$ corresponding to the term $I_2$ is even easier because we can estimate $S(R) \leq 4\pi \sigma^2$ and $\int_{\R^+} {L}(R,r) \diff R \leq C(\sigma)$. It follows that
$$
\int_{\R^+} \frac{\psi_n(R)}{R} \diff \widetilde{\mu_t}(R) \leq \|\gamma_1\|_{TV} + C(\sigma) \int_0^t \| \mu_s\|_{TV} \diff s.
$$
Application of monotone convergence theorem concludes the proof.
\end{proof}
\begin{lemma}[continuity of Duhamel's formula wrt parameters]\label{lem:continuity_duh_formula}
Let $\widetilde{\mu_t}$ and $\widetilde{\nu_t}$ be defined with \eqref{eq:duh_solution_concept_tildemu} and \eqref{eq:duh_solution_concept_tildenu} respectively. Then,
\begin{equation}\label{eq:lem:continuity_duh_formula}
\begin{split}
\| \widetilde{\mu_t} - \widetilde{\nu_t} \|_{BL^*,w}  \leq C_d \,  \left\|\frac{\gamma_1 - \gamma_2}{r\, \minT(1,r)} \right\|_{BL^*}  + C_d \, \int_0^t \| \mu_s - \nu_s \|_{BL^*,w} \diff s
\end{split}
\end{equation}
\end{lemma}

\begin{proof}
Thanks to  and \eqref{eq:duh_solution_concept_tildenu}, we write for all $\psi \in BL(\R^+)$ with $\| \psi \|_{BL} \leq 1$
\begin{equation*}
\int_{\R^+} \frac{\psi(R)}{R} \diff \widetilde{\mu_t}(R) = X_1 + Y_1, \qquad \int_{\R^+} \frac{\psi(R)}{R} \diff \widetilde{\nu_t}(R) =  X_2 + Y_2,
\end{equation*} 
where $X_1$ and $Y_1$ corresponds to the first and second summand in \eqref{eq:duh_solution_concept_tildemu} respectively; similarly for $X_2$ and $Y_2$. To estimate $\| \widetilde{\mu_t} - \widetilde{\nu_t} \|_{BL^*,w}$, we write
$
\int_{\R^+} \frac{\psi(R)}{R} \diff \left(\widetilde{\mu_t}-\widetilde{\nu_t}\right)(R) = (X_1 - X_2) + (Y_1 - Y_2)
$
and we estimate two differences separately.\\

\noindent \underline{Term $X_1 - X_2$.} With triangle inequality, this difference is controlled as
$$
\left|\int_{\R^+} \frac{\psi(R)}{R}  \, \mathcal{E}[\mu_{\bullet},R, t] \, \diff \left(\gamma_1 - \gamma_2\right)(R) \right| +  \left|\int_{\R^+} \frac{\psi(R)}{R}\,   \left(\mathcal{E}[\mu_{\bullet}, R, t] - \mathcal{E}[\nu_{\bullet}, R, t] \right) \, \diff \gamma_2(R) \right| =: Q_1 + Q_2.
$$
For term $Q_1$ we write
$$
Q_1 = \left|\int_{\R^+} \psi(R)  \, \minT(1,R) \, \mathcal{E}[\mu_{\bullet},R, t] \frac{\diff \left(\gamma_1 - \gamma_2\right)(R)}{R\, \minT(1,R)} \right|.
$$
Note that $\|\psi\|_{BL} \leq 1$ and $R \mapsto \minT(1,R) \, \mathcal{E}[\mu_{\bullet},R, t] \in BL(\R^+)$ with norm controlled with $C_d$ cf. Lemma \ref{lem:E_basic_estimates} \ref{E4}. Hence, product of these functions also belong to $BL(\R^+)$ with norm bounded by $2\,C_d$ cf. Lemma \ref{lem:prod_BL_bel}. It follows that
\begin{equation}\label{eq:estimateQ1_cont_duh}
Q_1 \leq C_d \, \left\|\frac{\gamma_1 - \gamma_2}{R\, \minT(1,R)} \right\|_{BL^*}.
\end{equation}
For term $Q_2$ we note that $\left|\mathcal{E}[\mu_{\bullet}, R, t] - \mathcal{E}[\nu_{\bullet}, R, t] \right| \leq \, \frac{C_d}{R} \, \int_0^t \| \mu_s - \nu_s \|_{BL^*,w} \diff s$ thanks to Lemma \ref{lem:E_basic_estimates} \ref{E2}. Hence,
\begin{equation}\label{eq:estimateQ2_cont_duh}
Q_2 \leq C_d \, \int_0^t \| \mu_s - \nu_s \|_{BL^*,w} \diff s \, \left\| \frac{\gamma_2}{R^2} \right\|_{TV} \leq C_d \, \int_0^t \| \mu_s - \nu_s \|_{BL^*,w} \diff s.
\end{equation}

\noindent \underline{Term $Y_1 - Y_2$.} To shorten formulas we introduce $\mathcal{E}^1(R,s):=\mathcal{E}[\mu_{\bullet}, R, t-s]$, $\mathcal{E}^2(R,s):=\mathcal{E}[\nu_{\bullet}, R, t-s]$. Aiming at triangle inequality again, we write 
\begin{align*}
|Y_1-Y_2| \leq & \, \left|\int_0^t \int_{\R^+}  \int_{\R^+} \frac{\psi(R)}{R}\,  \left(\mathcal{E}^1(R,s) - \mathcal{E}^2(R,s) \right) \, {L}(R,r) \,S(R) \diff \mu_s(r) \diff R \diff s \right| \\
&  + \left|\int_0^t \int_{\R^+} \int_{\R^+} \frac{\psi(R)}{R}\, \mathcal{E}^2(R,s) \,   {L}(R,r)\,S(R)  \diff (\mu_s -\nu_s)(r) \diff R \diff s \right| =:\, Z_1 + Z_2.
\end{align*}
For $Z_1$ we estimate $ \left|\mathcal{E}^1(R,s) - \mathcal{E}^2(R,s) \right| \leq  \frac{C_d}{R} \, \int_0^{t-s} \| \mu_u - \nu_u \|_{BL^*,w} \diff u$ using Lemma \ref{lem:E_basic_estimates}. Hence,
\begin{equation}\label{eq:estimate_aux_Z1}
Z_1 \leq C_d \, \int_0^{t} \| \mu_u - \nu_u \|_{BL^*,w} \diff u \, \int_0^t \int_{\R^+} \left[ \int_{\R^+} \frac{|\psi(R)|}{R^2} \, {L}(R,r) \,S(R) \diff R \right] \diff \mu_s(r)  \diff s.
\end{equation}
Now, we use $S(R) = 4\pi R^2$ and $L(R,r) = \frac{\widetilde{L}(R,r)}{Rr}$ cf. \eqref{eq:defLtilde}. Thanks to Lemma \ref{lem:prop_of_L} \ref{propH}, the integral $\int_{\R^+} \frac{|\psi(R)|}{R^2} \, \frac{\widetilde{L}(R,r)}{R} \, S_d(R) \diff R$ is bounded with $C_d$. Hence,
$$
\int_0^t \int_{\R^+} \left[ \int_{\R^+} \frac{|\psi(R)|}{R^2} \, {L}(R,r) \,S(R) \diff R \right] \diff \mu_s(r)  \diff s \leq 
C_d \int_0^t \int_{\R^+} \frac{1}{r} \diff \mu_s(r) \diff s \leq C_d \, t \, \| \mu_{\bullet} \|_{E_T},
$$
cf. \eqref{eq:norm_on_E}. From \eqref{eq:estimate_aux_Z1} we conclude that
\begin{equation}\label{eq:estimateZ1_cont_duh}
Z_1 \leq C_d \, \int_0^{t} \| \mu_s - \nu_s \|_{BL^*,w} \diff s.
\end{equation}
When it comes to term $Z_2$ we observe that $\frac{\psi(R)}{R}\, \mathcal{E}^2(R,s) \,   {L}(R,r)\,S(R) = {\psi(R)}\, \mathcal{E}^2(R,s) \,   \frac{\widetilde{L}(R,r)}{r} \,4\pi$. By Fubini Theorem,
$$
Z_2 = \int_0^t \int_{\R^+} \left[\int_{\R^+} \frac{\psi(R)}{R}\, \mathcal{E}^2(R,s) \,   {L}(R,r)\,S(R) \diff R \right] \diff (\mu_s -\nu_s)(r)  \diff s
$$
Hence, to bound the integrand in terms of $\|\mu_s - \nu_s\|_{BL^*,w}$, we only need to prove that the map $r \mapsto \int_{\R^+} \psi(R) \, \mathcal{E}^2(R,s) \,  {\widetilde{L}(R,r)} \diff R$ belongs to $BL(\R^+)$. As $\psi(R) \, \mathcal{E}^2(R,s)$ is bounded with $\|\psi\|_{\infty}$, this follows from Lemma \ref{lem:prop_of_L} \ref{propI}. Hence,
\begin{equation}\label{eq:estimateZ2_cont_duh}
Z_2 \leq C_d \, \int_0^{t} \| \mu_s - \nu_s \|_{BL^*,w} \diff s.
\end{equation}
Collecting estimates \eqref{eq:estimateQ1_cont_duh}, \eqref{eq:estimateQ2_cont_duh}, \eqref{eq:estimateZ1_cont_duh} and \eqref{eq:estimateZ2_cont_duh}, we conclude the proof of the lemma.
\end{proof}

\begin{lemma}[continuity of Duhamel's formula wrt time]\label{lem:cont_duh_time}
Let $\widetilde{\mu_t}$ be defined with \eqref{eq:duh_solution_concept_tildemu}. Then, under notation above,
\begin{equation}\label{eq:lem:cont_duh_time}
\left\| \widetilde{\mu_{t_1}} - \widetilde{\mu_{t_2}} \right\|_{BL^*,w} \leq C_d \, |t_1-t_2|.
\end{equation}
\end{lemma}
\begin{proof}
Thanks to \eqref{eq:duh_solution_concept_tildemu}, we have
\begin{equation*}
    \begin{split}
&\int_{\R^+} \frac{\psi(R)}{R} \diff \widetilde{\mu_{t_1}}(R) = \int_{\R^+} \frac{\psi(R)}{R}\,  \mathcal{E}[\mu_{\bullet}, R, {t_1}] \, \diff \gamma_1(R)\, + \\
& \quad \qquad \qquad + \int_0^{t_1}  \int_{\R^+} \int_{\R^+} \frac{\psi(R)}{R}\, \mathcal{E}[\mu_{\bullet}, R, t_1-u] \,  {L}(R,r) \, S(R) \diff \mu_u(r) \diff R \diff u =: G_1 + H_1,
\end{split}
\end{equation*}
and similar expression for $\widetilde{\mu_{t_2}}$ which we split for $G_2 + H_2$. As in Lemma \ref{lem:continuity_duh_formula}, we consider differences $G_1-G_2$ and $H_1-H_2$ separately. For term $G_1 - G_2$ we use $\left|\mathcal{E}[\mu_{\bullet}, R, t_1] - \mathcal{E}[\mu_{\bullet}, R, t_2]\right| \leq \frac{C_d}{R}|t_1-t_2|$ cf. Lemma \ref{lem:E_basic_estimates} \ref{E3} to compute
\begin{align*}
|G_1 - G_2| &\leq C_d \,\int_{\R^+} \frac{\psi(R)}{R}\,  \left|\mathcal{E}[\mu_{\bullet}, R, {t_1}] - \mathcal{E}[\mu_{\bullet}, R, {t_2}]\right| \, \diff \gamma_1(R) \leq \|\psi\|_{\infty} \, C_d\, |t_1-t_2| \, \left\| \frac{\gamma_1}{R^2} \right\|_{TV},
\end{align*}
cf. Lemma \ref{lem:prod_BL_bel} (A'). To estimate $H_1 - H_2$, we introduce auxillary notation $\mathcal{E}^{1}(u,R) = \mathcal{E}[\mu_{\bullet}, R, t_1-u]$, $\mathcal{E}^{2}(u,R) = \mathcal{E}[\mu_{\bullet}, R, t_2-u]$ and estimate
\begin{align*}
|H_1 - H_2| \leq &\, \int_{t_1}^{t_2} \int_{\R^+} \int_{\R^+} \frac{\psi(R)}{R}\, \mathcal{E}^{2}(u,R) \,  {L}(R,r) \, S(R) \diff \mu_u(r) \diff R \diff u\\
 &  +  \int_0^{t_1} \int_{\R^+} \int_{\R^+} \frac{\psi(R)}{R}\, (\mathcal{E}^{2}(u,R) - \mathcal{E}^{1}(u,R)) \,  {L}(R,r) \, S(R) \diff \mu_u(r) \diff R \diff u := K_1 + K_2. 
\end{align*}
To estimate $K_1$ we recall that ${L}(R,r) = \frac{\widetilde{L}(R,r)}{Rr}$ and $S(R)=4\pi R^2$ so that
\begin{multline*}
K_1 \leq 4\pi \int_{t_1}^{t_2} \int_{\R^+} \left[ \int_{\R^+} {\psi(R)}\, \mathcal{E}^{2}(u,R) \, \widetilde{L}(R,r)\diff R \right]  \frac{\diff \mu_u(r)}{r} \diff u \leq \\
\leq C_d  \int_{t_1}^{t_2} \int_{\R^+} \frac{\diff \mu_u(r)}{r} \diff u \leq C_d |t_1 - t_2| \| \mu_{\bullet} \|_{E_T} \leq C_d \,  |t_1-t_2|,
\end{multline*}
where we used Lemma \ref{lem:prop_of_L} \ref{propI} for the inner integral with respect to $R$. To estimate $K_2$ we use that $\left|\mathcal{E}^{1}(u,R) - \mathcal{E}^{2}(u,R)\right| \leq \frac{C_d}{R}|t_1-t_2|$ cf. Lemma \ref{lem:E_basic_estimates} \ref{E3}. Since $S(R) = 4\pi R^2$,
\begin{align*}
K_2 &\leq 4\pi C_d \int_0^{t_1} \int_{\R^+} \int_{\R^+} {\psi(R)}\, |t_1-t_2| \,  {L}(R,r)  \diff \mu_u(r) \diff R \diff u \\
&\leq C_d \,  |t_1-t_2| \int_0^{t_1} \int_{\R^+} \left[ \int_{\R^+} {\psi(R)}\, \,  \frac{\widetilde{L}(R,r)}{R} \diff R \right] \frac{\diff\mu_u(r)}{r}  \diff u \leq C_d \,  |t_1-t_2| \, \| \mu_{\bullet} \|_{E_T} \leq C_d \,  |t_1-t_2|,
\end{align*}
where we used Lemma \ref{lem:prop_of_L} \ref{propH} for the integral with respect to $R$.
\end{proof}

\noindent Now, we are in position to prove Theorem \ref{thm:well-posed_flat_metric_radial}. 

\begin{proof}[Proof of Theorem \ref{thm:well-posed_flat_metric_radial}]
Let $P:E_T \to E_T$ be defined with (RHS) of \eqref{eq:duh_solution_concept_TF}, i.e. for $\mu_{\bullet} \in E_T$ we put
\begin{equation*}
\begin{split}
&\int_{\R^+} \frac{\psi(R)}{R} \diff(P\mu_{\bullet})_t(R) = \int_{\R^+} \frac{\psi(R)}{R} \, e^{-\int_0^t \int_{\R^+} {L}(R, r) \diff \mu_s(r) \diff s }\diff\mu_0(R) \,+\\ &\qquad \qquad + \int_0^t \int_{\R^+} \int_{\R^+} \frac{\psi(R)}{R}\,  e^{-\int_0^{t-s} \int_{\R^+} {L}(R,r) \diff \mu_u(r) \diff u } \,   L(R,r) \, S(R) \diff \mu_s(r) \diff \lambda(R) \diff s .
\end{split}
\end{equation*}
We define
$$
K_T = \left\{\nu_{\bullet} \in E_T: \nu_0 = \mu_0, \sup_{t\in[0,T]} \|\nu_{t}\|_{BL^*,w} \leq 2\, \|\mu_0\|_{BL^*,w} \right\}
$$
and we want to find $T_{e}$ such that $P: K_{T_{e}} \to K_{T_{e}}$ is contractive. Lemma \ref{lem:aprioriestimateinweightedspace} implies
\begin{equation*}
\|(P\mu_{\bullet})_t\|_{BL^*,w} \leq \|\mu_0\|_{BL^*,w} + C(\sigma) \int_0^t \| \mu_s\|_{BL^*,w} \diff s \leq \|\mu_0\|_{BL^*,w} + C(\sigma) \,T_{e} \, 2 \, \|\mu_0\|_{BL^*,w}.
\end{equation*}
Hence, $T_{e} \leq 1/(2\,C(\sigma))$ and we only need to prove that $P$ is contractive. Using Lemma \ref{lem:continuity_duh_formula} with $\gamma_1 = \gamma_2 := \mu_0$, we discover that there is some constant $C_d$ depending on data such that
$$
\| (P\mu_{\bullet})_t - (P\nu_{\bullet})_t \|_{{BL^*,w}} \leq C_d \, \int_0^t \| \mu_s - \nu_s \|_{BL^*,w} \diff s \\ \leq C_d \, T_{e} \, \sup_{t \in [0,T_{e}]} \| \mu_{\bullet} - \nu_{\bullet} \|_{BL^*,w}. 
$$
Hence, Banach Fixed Point Theorem provides well-posedness for some small time depending on initial data $T_{e} = \min\left( \frac{1}{2 C(\sigma)}, \frac{1}{2 C_d}\right)$. But then Lemma \ref{lem:aprioriestimateinweightedspace} provides a priori estimate
\begin{equation}\label{eq:growth_of_solutions_explicit_constants}
\|\mu_t\|_{BL^*,w} \leq \| \mu_0\|_{BL^*,w} \, e^{C(\sigma) \, t}, \qquad \|\mu_t\|_{TV} \leq \| \mu_0\|_{TV} \, e^{C(\sigma) \, t}.
\end{equation}
This concludes the proof of well-posedness on arbitrary large interval of times.\\

\noindent To prove continuity estimate, we first observe that in view of \eqref{eq:growth_of_solutions_explicit_constants}, $\mu_t$ is bounded with some constant depending on $T$, $\sigma$ and the size of initial data with respect to $\| \cdot \|_{BL^*,w}$. Hence, constant $C_d$ appearing in Lemmas \ref{lem:continuity_duh_formula} and \ref{lem:cont_duh_time} can be bounded in terms of
$$
\left\|\frac{\mu_0}{r^2}\right\|_{BL^*}, \, \left\|\frac{\nu_0}{r^2}\right\|_{BL^*}, \,  \|\mu_0\|_{BL^*,w},  \, \|\nu_0\|_{BL^*,w}, \, L, \, \sigma, T.
$$
Applying Gronwall inequality to \eqref{eq:lem:continuity_duh_formula} we deduce \eqref{eq:important_continuity_estimate} while \eqref{eq:continuity_in_time_measure_solution} follows directly from \eqref{eq:lem:cont_duh_time}.
\end{proof}

%%%%%%%%%%%%%%%%%%%%%%%%%%%%%%%%%%%%%%%%%%%%%%%%%%%%%%%%%%%%%%%%%%%%%%%%%%%
%%%%%%%%%%%%%%%%%%%%%%%%%%%%%%%%%%%%%%%%%%%%%%%%%%%%%%%%%%%%%%%%%%%%%%%%%%%
%%%%%%%%%%%%%%%%%%%%%%%%%%%%%%%%%%%%%%%%%%%%%%%%%%%%%%%%%%%%%%%%%%%%%%%%%%%

\section{Solution to numerical scheme as a measure solution}\label{sect:infinite_speed_of_prop}

%%%%%%%%%%%%%%%%%%%%%%%%%%%%%%%%%%%%%%%%%%%%%%%%%%%%%%%%%%%%%%%%%%%%%%%%%%%
%%%%%%%%%%%%%%%%%%%%%%%%%%%%%%%%%%%%%%%%%%%%%%%%%%%%%%%%%%%%%%%%%%%%%%%%%%%
%%%%%%%%%%%%%%%%%%%%%%%%%%%%%%%%%%%%%%%%%%%%%%%%%%%%%%%%%%%%%%%%%%%%%%%%%%%

\noindent Now, we would like to find connection between numerical scheme \eqref{eq:ODE_num_scheme_INTRODUCTION} and measure solutions. The issue is that measure solutions to \eqref{eq:transform_coordinates_INTRODUCTION} are neither compactly supported nor concentrated at discrete points, a consequence of infinite speed of propagation. \\

\noindent First, we deal with the support of solution. The idea is to truncate the support and control the resulting error. The additional assumption made to handle these effects is that the initial datum has a certain decay at infinity. In particular, this is always satisfied for compactly supported initial datum.

\begin{assumption}[Decay of initial condition $\mu_0$]\label{ass:decay}
Let $M: \R^+ \to \R^+$ be a $C^1$ strictly increasing function with $\frac{M(r+\sigma)}{M(r)} \leq C_{M}(\sigma)$. We assume that $\mu_0 \in \mathcal{M}^+(\R^+)$ is such that 
$
\int_{\R^+} \frac{M(r)}{r} \diff \mu_0(r) < \infty.
$
\end{assumption}
\begin{remark}\label{rem:complicated_cond_on_M}
We list here three examples of functions $M$ satisfying condition $\frac{M(r+\sigma)}{M(r)} \leq C_{M}(\sigma)$.
\begin{itemize}
\item[(A)] $M_1(r) =e^r$ with
$
\frac{M_1(r+\sigma)}{M_1(r)} = e^{\sigma} =: C_{M_1}(\sigma). 
$
\item[(B)] $M_2(r) = (1+r)^k$ with $k\in \mathbb{N}$ and
$
\frac{M_2(r+\sigma)}{M_2(r)} \leq (1+\sigma)^k =: C_{M_2}(\sigma).
$
\item[(C)] If $M$ is as in Assumption \ref{ass:decay} then its truncation $M^k(r):= \begin{cases} M(r) &\mbox{ if } r\leq k \\
M(k) &\mbox{ if } r > k
\end{cases}$ satisfies $\frac{M^k(r+\sigma)}{M^k(r)} \leq \max(1,C_{M}(\sigma))$. This is not trivial only if $r + \sigma > k$ while $r \leq k$. But then, by monotonicity of $M$
$$
\frac{M^k(r+\sigma)}{M^k(r)} = \frac{M(k)}{M(r)} \leq \frac{M(r+\sigma)}{M(r)} \leq C_{M}(\sigma).
$$
\end{itemize}
\end{remark}
\begin{theorem}[Truncation of support]\label{lem:tail_control_solution}
Under Assumption \ref{ass:decay}, the mild measure solution of \eqref{eq:transform_coordinates_INTRODUCTION} with initial data $\mu_0$ satisfies
\begin{equation}\label{eq:propagation_of_moments}
\int_{\R^+} \frac{M(r)}{r} \diff \mu_t(r) \leq \int_{\R^+} \frac{M(r)}{r} \diff \mu_0(r) \, e^{C \,t}, \qquad C = C(\sigma)\,\max(1,C_M(\sigma)).
\end{equation}
In particular, we have an estimate on the tail of $\mu_t$:
$$
\left\| \mu_t|_{[0,R_0]} - \mu_t \right\|_{BL^*,w}  \leq \frac{1}{M(R_0)} \,  \left[\int_{\R^+} M(r) \diff \mu_0(r)\right] \, e^{C \,t},
$$
where $\mu_t|_{[0,R_0]}$ is restriction of $
\mu_t$ to the interval $[0,R_0]$.
\end{theorem}
\begin{proof}[Proof of Theorem \ref{lem:tail_control_solution}]
We want to use definition of measure solution \eqref{eq:duh_solution_concept_TF} with test function $\psi(R) = M(R)$ but this is not an admissible test function so we fix $k \in \mathbb{N}$ and consider truncation $M^k$ as in Remark \ref{rem:complicated_cond_on_M} (C). Hence, we have
\begin{equation}
    \begin{split}
&\int_{\R^+} \frac{M^k(R)}{R} \diff {\mu_t}(R) = \int_{\R^+} \frac{M^k(R)}{R} \, \mathcal{E}[\mu_{\bullet}, R, t] \diff \mu_0(R)\, + \\
& \qquad \qquad + \int_0^t  \int_{\R^+} \int_{\R^+} \frac{M^k(R)}{R}\, \mathcal{E}[\mu_{\bullet}, R, t-s] \,  {L}(R,r) \, S(R) \diff \mu_s(r) \diff R \diff s =: X + Y.
\end{split}
\end{equation} 
Term $X$ is trivially controlled with
$
|X| \leq \int_{\R^+} \frac{M^k(R)}{R} \diff \mu_0(R).
$
For term $Y$ we write ${L}(R,r) = \frac{\widetilde{L}(R,r)}{R\,r}$ and note that $\frac{S(R)}{R^2} = 4\pi$ so that
\begin{equation}\label{eq:proof_mono_lemma}
\begin{split}
|Y| & =  \int_0^t  \int_{\R^+} \int_{\R^+} \frac{M^k(R)}{R}\,  \mathcal{E}[\mu_{\bullet}, R, t-s] \,  \frac{\widetilde{L}(R,r)}{R\,r} \, S(R) \diff \mu_s(r) \diff R \diff s \\
& = 4\pi  \int_0^t  \int_{\R^+} \left( \int_{\R^+} \frac{M^k(R)}{M^k(r)} \,  \widetilde{L}(R,r) \diff R \right) \, \frac{M^k(r)}{r}  \diff \mu_s(r) \diff s
\end{split}
\end{equation}
As $\widetilde{L}(R,r)$ is supported only on $|R-r|\leq \sigma$, we can use monotonicity (not necessarily strict monotonicity) of $M^k$ and Assumption \ref{ass:decay} to estimate
\begin{equation}\label{eq:estimate_tail_toreferindiscr}
\begin{split}
\int_{\R^+} \frac{M^k(R)}{M^k(r)} \,  \widetilde{L}(R,r) \diff R \leq \int_{\R^+} &\frac{M^k(r+\sigma)}{M^k(r)} \,  \widetilde{L}(R,r) \diff R \leq \\ &\leq C_M(\sigma) \int_{\R^+} \widetilde{L}(R,r) \diff R \leq \max(1,C_M(\sigma)) \, C(\sigma),
\end{split}
\end{equation}
thanks to Lemma \ref{lem:prop_of_L} \ref{propI}. Therefore, \eqref{eq:proof_mono_lemma} implies
$$
|Y| \leq \max(1,C_M(\sigma)) \, C(\sigma)  \int_0^t \int_{\R^+} \frac{M(r)}{r}  \diff \mu_s(r) \diff s.
$$
so that using Gronwall's inequality we deduce
$$
\int_{\R^+} \frac{M^k(r)}{r} \diff \mu_t(r) \leq \int_{\R^+} \frac{M^k(r)}{r} \diff \mu_0(r) \, e^{C \,t}, \qquad C = C(\sigma)\,\max(1,C_M(\sigma)).
$$
But then, $\int_{\R^+} \frac{M^k(r)}{r} \diff \mu_0(r) \leq \int_{\R^+} \frac{M(r)}{r} \diff \mu_0(r)$ and $M_k(r) \to M(r)$ increasingly so monotone convergence theorem yields \eqref{eq:propagation_of_moments}. The estimate on tail follows from simple computation:
$$
\left\| \mu_t|_{[0,R_0]} - \mu_t \right\|_{BL^*,w}  = \left\| \mu_t|_{(R_0, \infty)} \right\|_{BL^*,w} = \int_{[R_0, \infty)} \frac{1}{r} \diff \mu_t(r) \leq \frac{1}{M(R_0)} \int_{\R^+} \frac{M(r)}{r} \diff \mu_t(r),
$$
where we used again monotonicity of $M$.
\end{proof}
\noindent Now, having support truncated at $R_0$, we consider equation in the space of measures
\begin{equation}\label{eq:approx_discrete_interactions}
\partial_t \mu_t^N = (\lambda_N \, S(R) - \mu_t^N) \int_{\R^+} {L}(R,r) \diff \mu_t^N(r)
\end{equation}
where $\lambda_N = \sum_{i=1}^N \frac{R_0}{N} \, \delta_{x_i}$ approximates Lebesgue measure $\lambda$ on $[0,R_0]$, $x_i := \frac{i}{N} \, R_0$ and \eqref{eq:approx_discrete_interactions} is understood in the usual mild sense as in Definition \ref{def:measure_solution}, i.e. for all test functions $\psi \in BL(\R^+)$ we have
\begin{equation}\label{eq:duh_solution_concept_TF_approxN}
\begin{split}
&\int_{\R^+} \frac{\psi(R)}{R} \diff\mu_t^N(R) = \int_{\R^+} \frac{\psi(R)}{R} \, e^{-\int_0^t \int_{\R^+} {L}(R, r) \diff \mu_s^N(r) \diff s }\diff\mu_0(R) \,+\\ &\qquad \qquad + \int_0^t \int_{\R^+} \int_{\R^+} \frac{\psi(R)}{R}\,  e^{-\int_0^{t-s} \int_{\R^+} {L}(R,r) \diff \mu_u^N(r) \diff u } \,   L(R,r) \, S(R) \diff \mu_s^N(r) \diff \lambda(R) \diff s .
\end{split}
\end{equation}
It turns out that solutions to numerical scheme \eqref{eq:ODE_num_scheme_INTRODUCTION} can be written as a measure solution to \eqref{eq:approx_discrete_interactions}. 

\begin{theorem}\label{thm:removing_local_interactions}
Let $\mu_0 \in \mathcal{M}^+_{w}(\R^+)$ be as in Theorem \ref{thm:well-posed_flat_metric_radial} and suppose it satisfies Assumption \ref{ass:decay}. Let $C_d$ be any constant depending on size of initial datum and model parameters. Let
$$
\mu_0^N := \sum_{i=1}^N m_i^N(0) \, \delta_{x_i},
$$
where $m_i^N(0) \in \R^+$ are given. Consider \eqref{eq:approx_discrete_interactions} with initial condition $\mu_0^N$ and $1 \leq R_0 \leq N$. Then, the unique non-negative measure solution is of the form
$
\mu_t^N = \sum_{i=1}^N m_i^N(t) \, \delta_{x_i}
$
where $m_i^N(t)$ solve numerical scheme \eqref{eq:ODE_num_scheme_INTRODUCTION}. Moreover:
\begin{enumerate}[label=(\Alph*)]
     \item\label{discr_2} we have estimates:
    $$
    \sup_{t \in [0,T]} \| \mu_t^N \|_{BL^*,w} \leq \| \mu_0^N\|_{BL^*,w} \, e^{C(\sigma) \, T} \leq C_d, 
    \qquad 
    \sup_{t \in [0,T]} \| \mu_t^N \|_{TV} \leq \| \mu_0^N\|_{TV} \, e^{C(\sigma) \, T} \leq C_d,
    $$
    
    \item\label{discr_tail} discretization $\mu_t^N$ satisfies the same decay estimate as $\mu_t$:
    $$
    \int_{\R^+} \frac{M(r)}{r} \diff \mu_t^N(r) \leq \int_{\R^+} \frac{M(r)}{r} \diff \mu_0^N(r) \, e^{C \,t}, \qquad C = C(\sigma)\,\max(1,C_M(\sigma)).
    $$
    
    \item\label{discr_4} if $\widetilde{\mu_t}$ is a measure solution to \eqref{eq:transform_coordinates_INTRODUCTION} with initial data $\mu_0^N$ we have
    \begin{equation*}
    \left\| \widetilde{\mu_t} - \mu_t^N \right\|_{BL^*,w} \leq  \frac{C_d\,R_0^2}{N} +  \frac{C_d}{M(R_0)}.
    \end{equation*}

\end{enumerate}
\end{theorem}

\begin{proof}
To prove representation formula for $\mu_t^N$, we first take any $\psi$ which is not supported at knots $\{x_i\}_{1\leq i \leq N}$. From definition of mild solution \eqref{eq:duh_solution_concept_TF_approxN} we obtain that $\int_{\R^+} \frac{\psi(R)}{R} \diff \mu_t^N(R) = 0$ which implies that any non-negative mild measure solution is supported on the knots. Hence, we only need to find $m_i^N(t)$. To do that, we take any $\psi$ which is supported only at one of the knots, say $x_i$. Because $x_i \neq 0$ 
\begin{equation*}
\begin{split}
&m_i^N(t) =  m_i^N(0)\,e^{-\int_0^t \int_{\R^+} {L}(x_i, r) \diff \mu_s^N(r) \diff s }\,+\phantom{\frac{R_0}{N}}\\ &\qquad \qquad +  \int_0^t \int_{\R^+} e^{-\int_0^{t-s} \int_{\R^+}   {L}(x_i,r) \diff \mu_u^N(r) \diff u } \,   \frac{R_0}{N} \, S(x_i) \, {L}(x_i,r) \diff \mu_s^N(r) \diff s.
\end{split}
\end{equation*}
Using support of $\mu^N_t$ we can write it as
\begin{equation*}
\begin{split}
&m_i^N(t) =  m_i^N(0)\,e^{-\int_0^t \sum_{j=1}^N {L}(x_i, x_j)\, m_j^N(s) \diff s }\,+\phantom{\frac{R_0}{N}}\\ &\qquad \qquad +  \int_0^t e^{-\int_0^{t-s} \sum_{j=1}^N {L}(x_i,x_j) \, m_j^N(u) \diff u } \,  \frac{R_0}{N} \, S(x_i) \, \sum_{j=1}^N {L}(x_i,x_j) \, m_j^N(s) \diff s.
\end{split}
\end{equation*}
This is precisely variation-of-constants formula for ODE \eqref{eq:ODE_num_scheme_INTRODUCTION}. \\

\noindent For the proof of \ref{discr_2}, we sum up equations \eqref{eq:ODE_num_scheme_INTRODUCTION} for $i=1, ..., N$ to obtain
\begin{equation*}
\begin{split}
\partial_t \sum_{i=1}^N \frac{m_i^N(t)}{x_i} &\leq \frac{R_0}{N} \, \sum_{i=1}^N \frac{4\pi x_i^2}{x_i} \sum_{j=1}^N {L}(x_i, x_j) \, m_j^N(t) =4\pi\, \frac{R_0}{N} \, \sum_{j=1}^N \sum_{i=1}^N 2\sigma^2 \mathds{1}_{|x_i - x_j| \leq \sigma} \, \frac{m_j^N(t)}{x_j},
\end{split}
\end{equation*}
where we used that $L(x_i, x_j) = \frac{\widetilde{L}(x_i, x_j)}{x_i\,x_j}$ is supported for $|x_i - x_j| \leq \sigma$ and that $\widetilde{L}(x_i, x_j) \leq 2\sigma^2$ cf. Lemma \ref{lem:prop_of_L} \ref{propB}, \ref{propD}. Fix $j \in \{1, ..., N\}$ and observe that condition $|x_i - x_j| \leq \sigma$ is equivalent to $|i - j| \leq \frac{N}{R_0}\,\sigma$. There are at most $2\,\frac{N}{R_0}\,\sigma + 1 \leq \frac{N}{R_0}\,(2\sigma + 1)$ points $i$ satisfying this condition. It follows that $\sum_{i=1}^N \mathds{1}_{|x_i - x_j| \leq \sigma} \leq \frac{N}{R_0}\,(2\sigma + 1)$ and this implies
$$
 \partial_t \sum_{i=1}^N \frac{m_i^N(t)}{x_i}  \leq 8\pi \sigma^2 \, \frac{R_0}{N} \, \frac{N}{R_0}\, (2\sigma + 1) \, \sum_{j=1}^N \frac{m_j^N(t)}{x_j} \leq C(\sigma) \, \sum_{j=1}^N \frac{m_j^N(t)}{x_j}.
$$
Application of Gronwall's inequality concludes the proof of the first estimate. To see the second one, we proceed in a similar manner, this time estimating
\begin{equation*}
\begin{split}
\partial_t \sum_{i=1}^N {m_i^N(t)} &\leq \frac{R_0}{N} \, \sum_{i=1}^N {4\pi x_i^2} \sum_{j=1}^N {L}(x_i, x_j) \, m_j^N(t) =4\pi\, \frac{R_0}{N} \, \sum_{j=1}^N \sum_{i=1}^N 2\sigma^2 \mathds{1}_{|x_i - x_j| \leq \sigma} \, {m_j^N(t)},
\end{split}
\end{equation*}
where we used $r^2 \, L(R,r) \leq C(\sigma)$ cf. Lemma \ref{lem:prop_of_L} \ref{propF}. We conclude in the same manner.\\

\noindent To prove \ref{discr_tail} we can repeat the proof of Theorem \ref{lem:tail_control_solution}. Indeed, the only problem one faces is to prove that $\int_{\R^+} \widetilde L(R,r) \diff \lambda_N(R) \leq C(\sigma)$ where $\lambda_N = \frac{R_0}{N} \sum_{i=1}^N \delta_{x_i}$ cf. \eqref{eq:estimate_tail_toreferindiscr}. However, we know that $\widetilde{L}(R,r)$ is supported for $|R-r| \leq \sigma$ and it is bounded with $C(\sigma)$ cf. Lemma \ref{lem:prop_of_L} \ref{propB}, \ref{propD}. For fixed $r$, there are at most $2\sigma/\left(\frac{N}{R_0}\right) + 1 \leq (2\sigma + 1) \, \frac{R_0}{N}$ points $x_i$ such that $|x_i - r| \leq \sigma$. Therefore,
$$
\int_{\R^+} \widetilde L(R,r) \diff \lambda_N(R) = 
\frac{R_0}{N} \sum_{i:|x_i-r|\leq \sigma} L(r,x_i) \leq C(\sigma) \, \frac{R_0}{N} \, (2\sigma + 1) \, \frac{R_0}{N} \leq C(\sigma).
$$

\iffalse
\noindent To get \ref{discr_3} (Lipschitz continuity in time), we first compute
$$
\| \mu_t^N - \mu_s^N \|_{BL^*(d_{\alpha})} \leq \sum_{i=1}^N |m_i^N(t) - m_i^N(s)|\\ \leq |t-s|\, \sum_{i=1}^N  \sup_{u \in [s,t]} \left| \partial_u m_i^N(u) \right|
$$
so we only need to study the sum appearing on the (RHS) above. From equation we get
\begin{multline*}
\sum_{i=1}^N  \left| \partial_t m_i^N(t) \right| \leq \,
\sum_{i=1}^N \left(\frac{R_0}{N} + m_i^N(t) \right) \sum_{j=1}^N \mathcal{L}(x_i, x_j) \, m_j^N(t)\\
=\,\frac{R_0}{N} \sum_{i = 1}^N \sum_{j=1}^N  \,  \mathcal{L}(x_i, x_j) \, m_j^N(t) + \sum_{i = 1}^N \sum_{j=1}^N   m_i^N(t) \,\mathcal{L}(x_i, x_j) \, m_j^N(t).
\end{multline*}
The first term is bounded by $C_1(\mathcal{L})\, (2\sigma+1) \|\mu_t^N\|_{TV}$ in view of \eqref{eq:estimate_counting_i} while the second is controlled by $C_1(\mathcal{L})\,\|\mu_t^N\|^2_{TV}$ which can be further estimated thanks to assertion \ref{discr_2}.\\
\fi

\noindent To prove \ref{discr_4}, we need to estimate $\int_{\R^+} \frac{\psi(R)}{R} \, \diff (\widetilde{\mu_t} - \mu_t^N)$ uniformly in $\psi \in BL(\R^+)$ with $\|\psi\|_{BL} \leq 1$. We let $\mathcal{E}(t) = e^{-\int_0^{t} \int_{\R^+} {L}(R,r) \diff \widetilde{\mu_u}(r) \diff u }$ and $\mathcal{E}^N(t) = e^{-\int_0^{t} \int_{\R^+} {L}(R,r) \diff \mu_u^N(r) \diff u}$. We have:
\begin{equation}\label{eq:proof_local_interactions_mu_t}
\begin{split}
\int_{\R^+} \frac{\psi(R)}{R} \diff\widetilde{\mu_t}(R) =& \int_{\R^+} \frac{\psi(R)}{R} \, \mathcal{E}(t) \diff\mu_0^N(R) \,+\\ & + \int_0^t \int_{\R^+} \int_{\R^+} \frac{\psi(R)}{R}\,  \mathcal{E}(t-s)\,   {L}(R,r) \, S(R) \diff \widetilde{\mu_s}(r) \diff \lambda(R) \diff s =: X_1 + Y_1.
\end{split}
\end{equation}
\begin{equation}\label{eq:proof_local_interactions_mu_t^N}
\begin{split}
\int_{\R^+} \frac{\psi(R)}{R} &\diff\mu_t^N(R) = \int_{\R^+} \frac{\psi(R)}{R} \, \mathcal{E}^N(t) \diff\mu_0^N(R) \,+\\ &  + \int_0^t \int_{\R^+} \int_{\R^+} \frac{\psi(R)}{R}\,  \mathcal{E}^N(t-s) \,   {L}(R,r) \, S(R) \diff \mu_s^N(r) \diff \lambda^N(R) \diff s =: X_2 + Y_2.
\end{split}
\end{equation}
We subtract these identities and study two terms independently. Before proceeding to computations, we note an auxiliary estimate that follows from \ref{E2} in Lemma \ref{lem:E_basic_estimates}:
\begin{equation}\label{auxillary_estimate_E} 
\left|\mathcal{E}(t) - \mathcal{E}^N(t)\right| \leq \frac{C(\sigma)}{R} \int_0^t \left\|\widetilde{\mu_s} - \mu_s^N \right\|_{BL^*,w} \diff s. 
\end{equation}

\noindent {\it Term with initial conditions $|X_1 - X_2|$.} We have $\| \psi \|_{\infty} \leq 1$ and $\| \mu_0^N\|_{TV} = \|\mu_0\|_{TV} \leq C_d$. Using \eqref{auxillary_estimate_E} we obtain
\begin{equation}\label{eq:approxestimate1_before}
|X_1 - X_2| \leq C(\sigma) \int_0^t \left\|\widetilde{\mu_s} - \mu_s^N \right\|_{BL^*,w} \diff s \, \left\| \frac{\mu_0^N}{R^2} \right\|_{TV}.
\end{equation}
We note that $ \left\| \frac{\mu_0^N}{R^2} \right\|_{TV} =  \left\| \frac{\mu_0}{R^2} \right\|_{TV}$ because with $A_1 = [0,x_1)$, $A_i = [x_{i-1},x_i)$ ($i = 2,..., n-1$) and $A_n = [x_{n-1}, x_n]$ we have
$$
\left\| \frac{\mu_0^N}{R^2} \right\|_{TV} = \sum_{i=1}^N \frac{\mu_0(A_i)}{x_i^2} \leq \int_{[0,R_0]} \frac{1}{R^2} \diff \mu_0(R) \leq \left\| \frac{\mu_0}{R^2} \right\|_{TV}
$$
therefore estimate \eqref{eq:approxestimate1_before} simplifies to
\begin{equation}\label{eq:approxestimate1}
|X_1 - X_2| \leq C_d \, \int_0^t \left\|\widetilde{\mu_s} - \mu_s^N \right\|_{BL^*,w} \diff s
\end{equation}
\noindent {\it Term with non-local interactions $|Y_1 - Y_2|$.} Aiming at triangle inequality, we write
\begin{align*}
    |Y_1 - Y_2| &\leq \int_0^t \int_{\R^+} \int_{\R^+} \frac{\psi(R)}{R}\, (\mathcal{E}(t-s) - \mathcal{E}^N(t-s))  \,   {L}(R,r) \, S(R) \diff \widetilde{\mu_s}(r) \diff \lambda(R) \diff s \\
     & + \int_0^t \int_{\R^+} \int_{\R^+} \frac{\psi(R)}{R}\,  \mathcal{E}^N(t-s)  \,   {L}(R,r) \, S(R) \diff \widetilde{\mu_s}(r) \diff (\lambda - \lambda_N)(R) \diff s \\
     & +  \int_0^t \int_{\R^+} \int_{\R^+} \frac{\psi(R)}{R}\,  \mathcal{E}^N(t-s)  \,   {L}(R,r) \, S(R) \diff (\widetilde{\mu_s} - \mu_s^N)(r) \diff \lambda_N(R) \diff s =: Z_1 + Z_2 + Z_3.
\end{align*}
\noindent {\it Term $Z_1$.} Note that $\|\psi\|_{\infty} \leq 1$ and $\| \widetilde{\mu_s} \|_{BL^*,w} \leq C_d$ due to estimate \eqref{eq:unif_bound_solutions}. As $S(R) = 4\pi R^2$, we use \eqref{auxillary_estimate_E} to obtain
\begin{equation}\label{eq:approxestimate2}
\begin{split}
Z_1 &\leq  4\pi \int_0^t \int_{\R^+} \left[\int_{\R^+}  \frac{\widetilde{L}(R,r)}{R} \diff \lambda(R)\right]\, \frac{\widetilde{\diff\mu_s}(r)}{r} \diff s \, \int_0^t \left\|\widetilde{\mu_s} - \mu_s^N \right\|_{BL^*,w} \diff s\\
&\leq  4\pi \,t \, C(\sigma) \, \sup_{s \in [0,t]} \| \mu_s\|_{BL^*,w}\, \int_0^t \left\|\widetilde{\mu_s} - \mu_s^N \right\|_{BL^*,w} \diff s \leq C_d \int_0^t \left\|\widetilde{\mu_s} - \mu_s^N \right\|_{BL^*,w} \diff s. 
\end{split}
\end{equation}
where in the second estimate we used Lemma \ref{lem:prop_of_L} \ref{propH} to bound integral $\int_{\R^+}  \frac{\widetilde{L}(R,r)}{R} \diff \lambda(R)$.\\

\noindent {\it Term $Z_2$.} First, we write
$$
Z_2 = 4\pi \int_0^t \int_{\R^+}  {\psi(R)}\,  \mathcal{E}^N(t-s) \left[\int_{\R^+} \widetilde{L}(R,r) \, \frac{\diff \widetilde{\mu_s}(r)}{r} \right] \diff (\lambda - \lambda_N)(R) \diff s
$$
We split integral with respect to $R$ for two subsets $R \in [0, R_0]$ and $R > R_0$ and denote the resulting integrals with $Z_2^{(1)}$ and $Z_2^{(2)}$. For $Z_2^{(1)}$, we consider the map
\begin{equation}\label{eq:functionG_lemma}
R \mapsto \mathcal{G}(R):= {\psi(R)}\,  \mathcal{E}^N(t-s) \left[\int_{\R^+} \widetilde{L}(R,r) \, \frac{\diff \widetilde{\mu_s}(r)}{r} \right].
\end{equation}
and Lemma \ref{lema:auxillary_computation_GlambdaN} below shows that it belongs to $BL\left[0, R_0 \right]$ with constant $C_d$. Thanks to Remark \ref{cor_approximation_lebesgue_measure} concerning discretisation of Lebesgue measure, we obtain
\begin{equation}\label{eq:estimateonZ22}
Z_2^{(1)} \leq C_d\, 
\left\| \lambda - \lambda_N  \right\|_{BL^*\left[0, R_0 \right]} \leq C_d \, \frac{R_0^2}{N}.
\end{equation}

\noindent To estimate $Z_2^{(2)}$ we observe that $L(R,r)$ is supported only for $|R-r|\leq \sigma$ cf. Lemma \ref{lem:prop_of_L} \ref{propD} so that we may restrict integral with respect to $r$ to the case $r \geq R_0 - \sigma$. Hence, 
$$
Z_2^{(2)} = 4 \pi \int_0^t \int_{r \geq R_0-\sigma} \left[\int_{R \geq R_0} \psi(R)\,  \mathcal{E}^N(t-s)  \,   \widetilde{L}(R,r) \diff \lambda(R) \right] \frac{\diff \widetilde{\mu_s}(r)}{r}  \diff s.
$$
The inner integral is controlled with $C(\sigma)$ due to Lemma \ref{lem:prop_of_L} \ref{propI}. For the outer one we use estimate on the tail from Theorem \ref{lem:tail_control_solution}, namely
\begin{equation}\label{eq:approxestimate4}
Z_2^{(2)} \leq \frac{C_d}{M(R_0-\sigma)} \int_0^t \int_{r \geq R_0-\sigma} M(r) \diff \widetilde{\mu_s}(r) \diff s \leq \frac{C_d\,C_{M}(\sigma)}{M(R_0)} \, \int_0^t C_M[\mu_s] \diff s \leq \frac{C_d}{M(R_0)}.
\end{equation}
where we used that $\frac{M(R_0)}{M(R_0 - \sigma)} \leq C_M(\sigma)$.

\noindent {\it Term $Z_3$.} First, we write $\frac{1}{R} L(R,r) S(R) = 4\pi \frac{\widetilde{L}(R,r)}{r}$. Hence, 
$$
Z_3 = 4\pi \int_0^t \int_{\R^+} \left[\int_{\R^+} \psi(R)\,  \mathcal{E}^N(t-s)  \,   \widetilde{L}(R,r) \, \diff \lambda_N(R)\right] \frac{\diff (\widetilde{\mu_s} - \mu_s^N)(r)}{r}  \diff s.
$$ 
We observe that the function 
\begin{equation}\label{eq:mapF_newlemma}
r \mapsto \mathcal{F}(r) := \int_{\R^+} \psi(R)\,  \mathcal{E}^N(t-s)  \,   \widetilde{L}(R,r) \, \diff \lambda_N(R)
\end{equation}
is in $BL(\R^+)$ with constant $C(\sigma)$ by Lemma \ref{lema:auxillary_computation_lambdaN} below. It follows that
\begin{equation}\label{eq:approxestimate5}
Z_3 \leq C(\sigma) \int_0^t \left\|\widetilde{\mu_s} - \mu_s^N \right\|_{BL^*,w} \diff s.
\end{equation}

\noindent {\it Conclusion.} Collecting estimates \eqref{eq:approxestimate1}, \eqref{eq:approxestimate2}, \eqref{eq:estimateonZ22}, \eqref{eq:approxestimate4} and \eqref{eq:approxestimate5} we obtain
\begin{equation*}
    \left\| \widetilde{\mu_t} - \mu_t^N \right\|_{BL^*,w} \leq   C_d \int_0^t  \left\| \widetilde{\mu_s} - \mu_s^N \right\|_{BL^*,w} \diff s \,+ \frac{C_d\,R_0^2}{N} +  \frac{C_d}{M(R_0)}.
\end{equation*}
Application of Gronwall concludes the proof of \ref{discr_4}.\\
\end{proof}

\begin{lemma}\label{lema:auxillary_computation_GlambdaN}
Let $\|\psi\|_{BL}\leq 1$. Under notation of the proof of Theorem \ref{thm:removing_local_interactions}, the function defined in \eqref{eq:functionG_lemma}
belongs to $BL\left[ 0, R_0 \right]$ with constant $C_d$.
\end{lemma}
\begin{proof}
First, function $\mathcal{G}$ is bounded with 
\begin{equation}\label{eq:estimate_infinity_G}
|\mathcal{G}(R)| \leq \| \psi \|_{\infty} \, \left| \mathcal{E}^N(t-s) \right| \, \|\widetilde{L}\|_{\infty} \, \sup_{s \in [0,t]} \left\| \widetilde{\mu_s}\right\|_{BL^*,w} \leq C(\sigma) \, C_d \leq C_d
\end{equation}
where we used $\| \psi \|_{\infty} \leq 1$, Lemma \ref{lem:E_basic_estimates} \ref{E1}, Lemma \ref{lem:prop_of_L} \ref{propB} and estimate \eqref{eq:unif_bound_solutions}. Aiming at application of Radamacher's theorem cf. Lemma \ref{lem:prod_BL_bel} (D), we compute its derivative and we want to prove that it is uniformly bounded. We have:
\begin{align*}
\partial_R \mathcal{G}(R) = \, & {\partial_R \psi(R)}\,  \mathcal{E}^N(t-s) \left[\int_{\R^+} \widetilde{L}(R,r) \, \frac{\diff \widetilde{\mu_s}(r)}{r} \right] + { \psi(R)}\,  \partial_R\mathcal{E}^N(t-s) \left[\int_{\R^+} \widetilde{L}(R,r) \, \frac{\diff \widetilde{\mu_s}(r)}{r} \right] \\
&+  { \psi(R)}\, \mathcal{E}^N(t-s) \left[\int_{\R^+}  \partial_R\widetilde{L}(R,r) \, \frac{\diff \widetilde{\mu_s}(r)}{r} \right] =: (A) + (B) + (C).
\end{align*}
Term $(A)$ is bounded by virtue of \eqref{eq:estimate_infinity_G} because $\|\psi\|_{BL} \leq 1$ implies $\|\partial_R \psi\|_{\infty} \leq 1$. For term $(C)$ we observe that Lemma \ref{lem:prop_of_L} \ref{propB} implies that $\left| \partial_R\widetilde{L}(R,r)\right| \leq C(\sigma)$ so that 
$$
|(C)| \leq C(\sigma)\, \sup_{s \in [0,t]} \left\| \widetilde{\mu_s}\right\|_{BL^*,w} \leq C_d.
$$

\noindent The most difficult part of the proof is analysis of term $(B)$. We observe that
$$
\partial_R \mathcal{E}^N(t-s) = \mathcal{E}^N(t-s) \,  \int_0^{t-s} \int_{\R^+} \partial_R L(R,r) \diff \mu_u^N(r) \diff u.
$$
Moreover, $| \mathcal{E}^N(t-s)| \leq 1$. Hence, we have
\begin{align*}
|(B)| &= { \psi(R)}\, \mathcal{E}^N(t-s) \, \left|\int_0^{t-s} \int_{\R^+} \partial_R L(R,r) \diff \mu_u^N(r) \diff u \right| \, \left[\int_{\R^+} \widetilde{L}(R,r) \, \frac{\diff \widetilde{\mu_s}(r)}{r} \right] \\
&\leq { \psi(R)}\, \left[\int_0^t \int_{\R^+} R\, \left| \partial_R L(R,r) \right| \diff \mu_u^N(r) \diff u \right] \, \left[\int_{\R^+} \frac{\widetilde{L}(R,r)}{R} \, \frac{\diff \widetilde{\mu_s}(r)}{r} \right]
\end{align*}
Now, we use \ref{propC} in Lemma \ref{lem:prop_of_L} to estimate $R\, \left| \partial_R L(R,r)\right|$:
$$
|(B)| \leq { \psi(R)}\, \left[\int_0^t \int_{\R^+} \left(\frac{2\sigma}{r} + L(R,r)\right) \diff \mu_u^N(r) \diff u \right] \, \left[\int_{\R^+} \frac{\widetilde{L}(R,r)}{R} \, \frac{\diff \widetilde{\mu_s}(r)}{r} \right] 
$$
Finally, we apply \ref{propE} in the Lemma \ref{lem:prop_of_L} to estimate $\frac{\widetilde{L}(R,r)}{R} \leq C(\sigma)$ and estimate on $\|\mu_s^N\|_{BL^*,w}$ from Theorem \ref{thm:removing_local_interactions} \ref{discr_2} to obtain
\begin{align*}
|(B)| &\leq C_d \,  \left[\int_0^t \int_{\R^+} \left(\frac{2\sigma}{r} + \frac{\widetilde{L}(R,r)}{R\,r}\right) \diff \mu_u^N(r) \diff u \right] \leq  C_d \,\int_0^t \int_{\R^+} \frac{\diff \mu_s^N(r)}{r} \leq C_d.
\end{align*}
\end{proof}

\begin{lemma}\label{lema:auxillary_computation_lambdaN}
Let $\|\psi\|_{BL}\leq 1$. Under notation of the proof of Theorem \ref{thm:removing_local_interactions}, the function defined in \eqref{eq:mapF_newlemma}
belongs to $BL(\R^+)$ with constant $C(\sigma)$.
\end{lemma}
\begin{proof}
Recall that ${L}$ is supported for $|R-r|\leq \sigma$ and $|L| \leq 2\sigma^2$. Hence, when $r$ is fixed, there are at most $2\sigma/(R_0/N) + 1= 2\sigma \frac{N}{R_0} + 1 \leq (2\sigma + 1) \frac{N}{R_0}$ points $x_i = \frac{i}{N}R_0$ in the interval $|R-r|\leq \sigma$. It follows that
$$
|\mathcal{F}(r)| \leq 2\sigma^2\|\psi\|_{\infty} \,  \sum_{i=1}^N \frac{R_0}{N} \,  \mathds{1}_{|x_i - r|} \leq 2\sigma^2\,(2\sigma + 1).
$$
Similarly, as $\psi(R)\,  \mathcal{E}^N(t-s)$ does not depend on $r$, $|\psi(R)\,  \mathcal{E}^N(t-s)| \leq 1$ and the map $r\mapsto \widetilde{L}(R,r)$ is in $BL(\R^+)$ with constant $C(\sigma)$ cf. Lemma \ref{lem:prop_of_L} \ref{propB}, we have for some $r_1$ and $r_2$:
\begin{align*}
|\mathcal{F}(r_1) - \mathcal{F}(r_2)| &= \left| \int_{\R^+} \psi(R)\,  \mathcal{E}^N(t-s)  \,   \left(\widetilde{L}(R,r_1) - \widetilde{L}(R,r_2)\right) \diff \lambda_N(R) \right|\\
&\leq \sum_{i=1}^N \frac{R_0}{N} \,C(\sigma) \,\left(  \mathds{1}_{|x_i - r_1|\leq\sigma} + \mathds{1}_{|x_i - r_2| \leq \sigma}\right) |r_1 - r_2| \leq 2\,(2\sigma + 1) \, C(\sigma) \, |r_1 - r_2|.
\end{align*}
\end{proof}

%%%%%%%%%%%%%%%%%%%%%%%%%%%%%%%%%%%%%%%%%%%%%%%%%%%%%%%%%%%%%%%%%%%%%%%%%%%
%%%%%%%%%%%%%%%%%%%%%%%%%%%%%%%%%%%%%%%%%%%%%%%%%%%%%%%%%%%%%%%%%%%%%%%%%%%
%%%%%%%%%%%%%%%%%%%%%%%%%%%%%%%%%%%%%%%%%%%%%%%%%%%%%%%%%%%%%%%%%%%%%%%%%%%

\section{Convergence result: proof of Theorem \ref{thm:main_res_convergence}}\label{sect:main_convergence}

%%%%%%%%%%%%%%%%%%%%%%%%%%%%%%%%%%%%%%%%%%%%%%%%%%%%%%%%%%%%%%%%%%%%%%%%%%%
%%%%%%%%%%%%%%%%%%%%%%%%%%%%%%%%%%%%%%%%%%%%%%%%%%%%%%%%%%%%%%%%%%%%%%%%%%%
%%%%%%%%%%%%%%%%%%%%%%%%%%%%%%%%%%%%%%%%%%%%%%%%%%%%%%%%%%%%%%%%%%%%%%%%%%%

\noindent Let $\mu_0$ be a measure generated by density $p(r,0) = 4\pi r^2 \, n_0(r)$, i.e. $\mu_0(A) = \int_{A} p(r,0) \diff r$. Then, $\mu_0$ satisfies Assumption \ref{ass:decay} and assumption of Theorem \ref{thm:well-posed_flat_metric_radial} because $n_0$ is bounded and compactly supported cf. Remark \ref{rem:ass_initial_condition_theorem}. In what follows,
\begin{itemize}
    \item $\mu_t$ is the unique measure solution to \eqref{eq:transform_coordinates_INTRODUCTION} with initial condition $\mu_0$,
    \item $\mu_{t}^N = \sum_{i=1}^N m_i(t) \, \delta_{x_i}$ is the solution to the numerical scheme \eqref{eq:ODE_num_scheme_INTRODUCTION} with initial condition $\mu_0^N = \sum_{i=1}^N m_i(0) \, \delta_{x_i}$,
    \item $\widetilde{\mu_t}$ is the unique measure solution to \eqref{eq:transform_coordinates_INTRODUCTION} with initial condition $\mu_0^N$.
\end{itemize}
According to Theorem \ref{thm:well-posed_flat_metric_radial} we have
$$
\left\|\mu_t - \widetilde{\mu_t} \right\|_{BL^*,w} \leq C \, \left\|\frac{\mu_0 - \mu_0^{N}}{r\, \minT(1,r)}\right\|_{BL^*}
$$
for some constant $C$ depending on data and initial condition $\mu_0$. Now, we observe that if $n_0$ is supported on $[0,\gamma]$ then 
$$
\int_{\R^+} \frac{p_0(r)}{r\,\minT(1,r)} \diff r =  \int_{\R^+} \frac{4\pi r^2 \, n_0(r)}{r\,\minT(1,r)} \diff r \leq 
\| n_0\|_{\infty} \, \gamma \, \mbox{max}(1,\gamma) < \infty,
$$
where we applied estimate $\frac{r}{\minT(1,r)} \leq \mbox{max}(1,r)$. Hence, we may apply Lemma \ref{lem:approx_distr} with $f(r) = r\,\minT(1,r)$ to obtain
\begin{equation}\label{eq:final_proof_first_estimate}
\left\|\mu_t - \widetilde{\mu_t} \right\|_{BL^*,w} \leq C \, \left\|\frac{\mu_0 - \mu_0^{N}}{r\, \minT(1,r)}\right\|_{BL^*} \leq C\, \frac{R_0}{N}.
\end{equation}
Moreover, directly from Theorem \ref{thm:removing_local_interactions} \ref{discr_4} we obtain
\begin{equation}\label{eq:final_proof_second_estimate}
    \left\| \widetilde{\mu_t} - \mu_t^N \right\|_{BL^*,w} \leq  C\,\frac{R_0^2}{N} +  C\, e^{-R_0}.
\end{equation}
where we chose $M(R) = e^{-R}$ for simplicity (in fact, any $M$ from Remark \ref{rem:complicated_cond_on_M} can be chosen as initial condition is compactly supported). Combining \eqref{eq:final_proof_first_estimate} and \eqref{eq:final_proof_second_estimate} we deduce
$$
\left\|\mu_t - {\mu_t^N} \right\|_{BL^*,w} \leq C\, \frac{R_0}{N} + C\,\frac{R_0^2}{N} +  C\, e^{-R_0} \leq C\,\frac{R_0^2}{N} +  C\, e^{-R_0}
$$
because $R_0$ is large (say $R_0 \geq 1$). This concludes the proof.

%%%%%%%%%%%%%%%%%%%%%%%%%%%%%%%%%%%%%%%%%%%%%%%%%%%%%%%%%%%%%%%%%%%%%%%%%%%
%%%%%%%%%%%%%%%%%%%%%%%%%%%%%%%%%%%%%%%%%%%%%%%%%%%%%%%%%%%%%%%%%%%%%%%%%%%
%%%%%%%%%%%%%%%%%%%%%%%%%%%%%%%%%%%%%%%%%%%%%%%%%%%%%%%%%%%%%%%%%%%%%%%%%%%

\section{Comment on 2D case and consequences on convergence}\label{sect:2Dcase}

%%%%%%%%%%%%%%%%%%%%%%%%%%%%%%%%%%%%%%%%%%%%%%%%%%%%%%%%%%%%%%%%%%%%%%%%%%%
%%%%%%%%%%%%%%%%%%%%%%%%%%%%%%%%%%%%%%%%%%%%%%%%%%%%%%%%%%%%%%%%%%%%%%%%%%%
%%%%%%%%%%%%%%%%%%%%%%%%%%%%%%%%%%%%%%%%%%%%%%%%%%%%%%%%%%%%%%%%%%%%%%%%%%%

\noindent In this section we briefly discuss the case $d=2$. This time, the change of variables yields the following result.
\begin{thm}[Radial change of coordinates in 2D]\label{thm2:radial_change_in2D}
Let $d=2$ and let $n(x,t)$ be the solution to \eqref{non-local_proliferation} with radially symmetric initial condition $n_0(x)$. Then the radial density $p(R,t)$ defined with
$$
    p(R,t) = 2\pi R\, n((0,0,R),t), \qquad \qquad p_0(R) = 2\pi R\, n_0((0,0,R)).
$$
satisfies equation
\begin{equation}\label{non-local_proliferation_polar_2D}
\partial _t p(R,t) \  =\  \left( 2\pi R \ -\  p(R,t) \right) \,  \int_0^\infty L(R,r)\, p(r,t)  \diff r
\end{equation}
where the radial interaction kernel $L$ is given by
\begin{equation}\label{eq:defL_2d}
L(R,r) =  \frac{1}{\pi^2 \sigma^2} \, \left[\frac{\pi}{2} - \arcsin \max \left(\frac{R^2+r^2-\sigma^2}{2Rr}, -1\right)  \right] \, \mathds{1}_{|R-r|\leq \sigma}.
\end{equation}
\end{thm}
\noindent Proof of Theorem \ref{thm2:radial_change_in2D} is deferred to the end of this section. Here, we briefly discuss how to adapt three dimensional proof to the two dimensional case. \\

\noindent First, the resulting kernel $L(R,r)$ cannot be expected to be Lipschitz continuous, even away of $R,r =0$, because function $x \mapsto \arcsin(x)$ is not Lipschitz (its derivative blows up at $x = \pm 1$). However, it is well-known that $\arcsin(x)$ is $1/2$-H\"older continuous. In another words, it is Lipschitz continuous with respect to the metric $d_{1/2}(x,y) = |x-y|^{1/2}$.\\

\noindent To exploit this fact we use the theory of measure spaces on general metric spaces cf. \cite{our_book_ACPJ}. Mimicking definitions \eqref{eq:def_BL}--\eqref{eq:def_BLnorm}, we define
\begin{equation*}
	BL_{1/2}(\R^+)=\left\{f:\R^+ \to \R \mbox{ is continuous and } \|f\|_{\infty}<\infty, |f|_{1/2}<\infty\right\},
\end{equation*}
 where  
 \begin{equation}
 \|f\|_{\infty}=\underset{x\in \R^+}{\sup}\,|f(x)|, \qquad \qquad |f|_{1/2}=\underset { x\neq y}{\sup}\, \frac {|f(x)-f(y)|}{|x-y|^{1/2}}.
 \end{equation}
Space $BL_{1/2}(\R^+)$ is equipped with the norm 
\begin{equation}
\|f\|_{BL,\,1/2} = \max\left(\|f\|_{\infty}, \, |f|_{1/2}\right) \leq \|f\|_{\infty} + |f|_{{1/2}}.
\end{equation}
Similarly as in \eqref{defeq:flatnorm_weighted}, we define weighted flat norm:
\begin{equation}\label{defeq:flatnorm_weighted_2D}
		\|\mu\|_{BL^*,1/2,w}:= \sup\left\{\int_{\R^+} \! \frac{\psi(r)}{\sqrt{r}}  \,\mathrm{d}\mu(r) : \psi \in BL_{1/2}(\R^+), \|\psi\|_{BL,\,1/2} \leq 1\right\}.
\end{equation}
This time we scale with $\sqrt{r}$ rather than $r$ because function $r \mapsto \sqrt{r}$ is 1/2-H\"older continuous so it will interplay nicely with test functions in $BL_{1/2}(\R^+)$.\\

\noindent Note that the distance between two Dirac masses in the numerical scheme equals $R_0/N$. By virtue of Lemma \ref{lem:approx_distr}, we see that arbitrary measure can be also approximated with finite combinations of Dirac masses with respect to $\|\cdot\|_{BL^*,1/2,w}$ norm with an error of size ${R_0^{1/2}}/{N^{1/2}}$. Therefore, in two dimensions we should expect that the methods of this paper yields estimate
$$
\left\|p(\cdot,t) -  \mu_{t}^N \right\|_{BL^*,1/2,w} \leq C\,\frac{R_0}{\sqrt{N}} +  C\, e^{-R_0},
$$
contrary to \eqref{eq:error_in_main_res}. This is indeed verified in the numerical simulations in Section \ref{sect:sim_res}.

\begin{proof}[Proof of Theorem \ref{thm2:radial_change_in2D}]
The kernel is given by 
\begin{equation}\label{kernel2D}
K(r)=\frac{1}{\pi\,\sigma^{2}}\,\mathds{1}_{\left[0,\sigma\right]}(r)
\end{equation}
Similarly to the 3D case, we let $p(R,t) = 2\pi R\,n(x,t)$, where $R = | x | = (x_1^2 + x_2^2 )^{ 1/2}$. The convolution $k*n$ being also a radially symmetric function given by 
\begin{equation}\label{splot1_2D}
\begin{split}
&k*n(x,t) = k*n((0,R),t) = \\ 
&\qquad=\int_{\R^2}K\left(\left((0-y_1)^2+(R-y_2)^2\right)^{1/2}\right)n((y_1,y_2),t) \diff y  \\ 
&\qquad=\int_{\R^2}K\left(\left(y_1^2+y_2^2+R^2-2\,R\,y_2\right)^{1/2}\right)p\left(\left(y_1^2+y_2^2\right)^{1/2},t\right)\, \frac{1}{2\pi\sqrt{y_1^2 + y_2^2}}\diff y. 
\end{split}
\end{equation}

\noindent To convert \eqref{non-local_proliferation} to polar coordinates we substitute
\begin{equation}\label{eq:change_of_var_R2}
y_1 = r \cos{\alpha}, \qquad \qquad
y_2 = r \sin{\alpha},
\end{equation}
where $r>0$ and $0 \leq \alpha \leq 2\pi$. The Jacobian determinant of the change of variables in \eqref{eq:change_of_var_R2} is equal to $r$. Using
$$
r^2=y_1^2+y_2^2, \qquad \qquad 2\,R\,y_2=2\,R\,r\sin\alpha
$$
to \eqref{splot1_2D}, substituting $u = r^2+R^2-2\,R\,r\sin\alpha$ so that $\sin\alpha = \frac{r^2+R^2-u}{2\,R\,r}$ and using the fact that \eqref{kernel2D} is an indicator function we obtain 
\begin{align*}
& \int_0^\infty \int_0^{2\pi}  K\left((r^2+R^2-2\,R\,r\,\sin\alpha)^{1/2}\right) \, \frac{1}{2\pi\,r}r \, p(r,t) \diff \alpha \diff r\\
& \quad =\frac{2}{2\pi}\int_0^\infty \int_{(R-r)^2}^{(R+r)^2} \frac{K\left(u^{1/2}\right) p(r,t)}{\sqrt{4\,R^2\,r^2-(R^2+r^2-u)^2}} \diff u \diff r \\
& \quad = \frac{1}{\pi^2 \, \sigma^2} \int_{|R-r|\leq\sigma}\int_{(R-r)^2}^{\min\{\sigma^2,\,(R+r)^2\}} \frac{p(r,t)}{\sqrt{4\,R^2\,r^2-(R^2+r^2-u)^2}} \diff u \diff r. 
\end{align*}
%\noindent Like in the 3D case we introduce an auxiliary function $L:[0,\infty)\to \R$ to be a function such that for the kernel $K$ given by \eqref{kernel2D} we get 
%\begin{equation}\label{\mathcal{L}_2d}
%    L_u(R,r,v)=K(v^{1/2})\frac{1}{\sqrt{4R^2r^2-(R^2+r^2-v)^2}}.
%\end{equation}

\noindent Again, substituting $w=\frac{R^2+r^2-u}{2Rr}$ and integrating with respect to $w$ we obtain 
\begin{align*}
\frac{1}{\pi^2 \, \sigma^2}  \int_{|R-r|\leq\sigma} \left[  \frac{\pi}{2} - \arcsin \max \left( \frac{R^2+r^2-\sigma^2}{2Rr}, -1 \right) \right] p(r,t) \diff r.
\end{align*}
%\begin{equation}
%  K*p(R,t)=\frac{\pi}{R}\int_0^\infty \bigg(L((R+r)^2)-L((R-r)^2)\bigg)\,p(r,t)\,r \diff r, 
%\end{equation}
\end{proof}

%%%%%%%%%%%%%%%%%%%%%%%%%%%%%%%%%%%%%%%%%%%%%%%%%%%%%%%%%%%%%%%%%%%%%%%%%%%
%%%%%%%%%%%%%%%%%%%%%%%%%%%%%%%%%%%%%%%%%%%%%%%%%%%%%%%%%%%%%%%%%%%%%%%%%%%
%%%%%%%%%%%%%%%%%%%%%%%%%%%%%%%%%%%%%%%%%%%%%%%%%%%%%%%%%%%%%%%%%%%%%%%%%%%

\section{Simulation results}\label{sect:sim_res}

\noindent This section presents the computational results illustrating the theoretical estimation concerning the orders of convergence of the numerical method used to solve  \eqref{eq:transform_coordinates_INTRODUCTION} and \eqref{non-local_proliferation_polar_2D}. As the Runge-Kutta method requires higher regularity of the right-hand side, to deal with the resulting ODE system we use the explicit Euler scheme. The codes used to perform presented simulations are available in the GitHub repository \cite{szymanska_github}.\\

\noindent While investigating the mentioned orders of convergence we assume each time $\Delta t = \Delta r$. We divide simulation time $T$ into $T/\Delta t$ time steps and the spatial domain $R_0$ into $R_0/\Delta t$ space cells. Then $\mu_t^{N_{\Delta t}}$ denotes the numerical solution obtained for time t, and $N_{\Delta t}$ space cells, where $N_{\Delta t} = R_0/\Delta t$. We define the relative error of the numerical solution $\mu_t^{N_{\Delta t}}$  as the flat norm of the distance between the solution $\mu_t^{N_{\Delta t}}$ and $\mu_t^{N_{\Delta 2t}}$, that is 
\begin{equation}\label{blad_zbieznosci}
\textrm{Err}(\Delta t)= \|\mu_t^{N_{\Delta t}} - \mu_t^{N_{\Delta 2t}}\|_{BL^*,w},
\end{equation}
where the norm is defined by \eqref{defeq:flatnorm_weighted}. Then, the rate of convergence, denoted by $q$, is given by the following formula
\begin{equation}\label{rzad_zbieznosci}
q := \lim_{\Delta t \to 0} q_{\Delta t}, \qquad q_{\Delta t} :=  \frac{\log (\textrm{Err}(2\Delta t)) / \textrm{Err}(\Delta t)) }{\log 2}.
\end{equation}
In the following computations, flat norm and Wasserstein distance were computed using standard algorithms cf. \cite[Sections 3.1, 3.3]{jablonski2013efficient} or \cite[Chapters 4.2-4.4]{our_book_ACPJ}.\\

\noindent {\textbf{Three dimensional case.}} \\
\noindent \tabref{q3D} presents the obtained relative error and the order of convergence of the numerical scheme applied to solve \eqref{eq:transform_coordinates_INTRODUCTION}. The presented simulations were conducted on the spatial domain $R_0$ equal to 2 and smallest $\Delta t = 1.5625 \cdot 10^{-5}$, which gives us $1.28\cdot 10^{5}$ mass points for the smallest discretisation, whereas the simulation time $T$ was equal to 10. We formulated the initial condition to be the same as the one adopted in our companion paper on modelling cells’ proliferation within a solid tumour, see \cite[Eq. 15]{szymanska2021} for a precise explanation of the launched formula. In short, we assume that the initial condition is given by 
\begin{equation}\label{initial_condition}
p(r,0) = 4 \pi r^2 \Big( 1 - \Big(\frac{r}{\tilde{\sigma_i}}\Big)^{\tilde{q}} \Big)\mathds{1}_{\left[0,\tilde{\sigma_{i}}\right]}(r),
\end{equation}
where $\tilde{\sigma_i} = 0.79$ and $\tilde{q} = 13$ are chosen so to match the experimental data we used to estimate parameters of the model. The original model includes an additional parameter describing the proliferation rate. More precisely, instead of \eqref{eq:defL} we consider now
\begin{equation}\label{eq:defL_bis}
L(R,r) =\frac{3\alpha}{16\,\pi\,\sigma^{3}} \, \frac{\min\{(R+r)^2,\sigma^2\}-\min\{(R-r)^2,\sigma^2\}}{R \, r}
\end{equation}
where both $\alpha$ and $\sigma$ together with $\sigma_i$ from the initial condition were subjects of performed parameter estimation \cite{szymanska2021}. Within the present computations, $\alpha = 0.5$ and $\sigma = 0.04$ are chosen to get solutions sufficiently distant from the initial condition to investigate the orders of convergence rather than to meet the experimental data.
\begin{table}[ht]
\begin{tabular}{r|c|c}
\hline
$\Delta t = \Delta r$ & $\textrm{Err}\,(\Delta t)$ &$q_{\Delta t}$
\\\hline \hline
1.5625$\cdot 10^{-5}$ & 3.86073194$\cdot 10^{-5}$ & \bf{--} 
\\\hline
3.125$\cdot 10^{-5}$ & 7.79035932$\cdot 10^{-5}$ & 1.0128154842837573
\\\hline
6.25$\cdot 10^{-5}$ & 1.574836435$\cdot 10^{-4}$ & 1.0154402184761508
\\\hline
1.25$\cdot 10^{-4}$ & 3.259716168$\cdot 10^{-4}$ & 1.0495443548927024
\\\hline
2.5$\cdot 10^{-4}$ & 6.788816649$\cdot 10^{-4}$ & 1.0584137711334787 
\\\hline
5.0$\cdot 10^{-4}$ & 1.534610347$\cdot 10^{-3}$ & 1.1766403603675424
\\\hline
1.0$\cdot 10^{-3}$ & 3.775119532$\cdot 10^{-3}$ & 1.2986499380494803 
\\\hline
\end{tabular}
\caption{Error computed in flat metric $\textrm{Err}\,(\Delta t)$ \eqref{blad_zbieznosci} together with corresponding order of convergence \eqref{rzad_zbieznosci} for the model given by  \eqref{eq:transform_coordinates_INTRODUCTION}.}\label{q3D}
\end{table}

\phantom{...}\\
\noindent {\textbf{Two dimensional case.}}\\
\noindent Computation of the flat metric with respect to H\"older metric is slightly difficult. Therefore, we introduce the following distance on $\mathcal{M}(\R^+)$:
\begin{equation}\label{metryka_rho}
\rho\big(\mu_1,\mu_2\big)=\min \Big\{ M_{\mu_{1}}, M_{\mu_{2}}\Big\}W_1\big(\tilde{\mu}_1,\tilde{\mu}_2\big) + \Big|M_{\mu_{1}} - M_{\mu_{2}}\Big|,
\end{equation}
where $M_{\mu_{i}}=\int_{\R^+} \! \,\mathrm{d}\mu_{i} \neq 0$, $\tilde{\mu}_i = \mu_i/M_{\mu_i}$ and $W_1$ is the usual Wasserstein distance with respect to $\frac{1}{2}$-H\"older metric, i.e.
$$
W_1\big(\tilde{\mu}_1,\tilde{\mu}_2\big) = \sup\left\{ \int_{\R^+} \psi(r) \diff(\mu_1 - \mu_2)(r) \quad \Big| \quad \psi: \R^+ \to \R, |\psi|_{1/2} \leq 1\right\} 
$$
It is well-known \cite{MR3986559,carrillo2014} that for measures defined on a bounded intervals of $\R^+$ metric $\rho$ is equivalent with the flat norm. In the following computations we use $\rho$ because there is a simple linear algorithm to compute $W_1$ cf. \cite[Sections 3.1, 3.3]{jablonski2013efficient} or \cite[Chapters 4.2-4.4]{our_book_ACPJ}.\\

\noindent We define the relative error as the distance between weighted solutions $\mu_t^{N_{\Delta t}}$ and $\mu_t^{N_{\Delta 2t}}$:
\begin{equation}\label{blad_zbieznosci_2D}
\textrm{Err}(\Delta t)= 
%\|\mu_t^{N_{\Delta t}} - \mu_t^{N_{\Delta 2t}}\|_{BL^*,w},
\rho\left( \frac{\mu_t^{N_{\Delta t}}}{\sqrt{r}}, \frac{\mu_t^{N_{\Delta 2t}}}{\sqrt{r}}\right).
\end{equation}
Here, $\frac{\mu_t^{N_{\Delta t}}}{\sqrt{r}}$ is a weighted measure as in \eqref{eq:def_weighted_measure} defined with the formula
$$
\frac{\mu_t^{N_{\Delta t}}}{\sqrt{r}}(A) = \int_{A} \frac{1}{\sqrt{r}} \diff \mu_t^{N_{\Delta t}}(r)
$$
and similarly for $\frac{\mu_t^{N_{\Delta 2t}}}{\sqrt{r}}$. Then, the order of convergence is computed as in \eqref{rzad_zbieznosci}.\\

%%%%%%%%%%%%%%%%%%%%%%%%%%
%%%% TRZECI BELKOT
%%%%%%%%%%%%%%%%%%%%%%%%%%

\iffalse 
\noindent The initial condition for the 2D case is similar to \eqref{initial_condition} with the only adjustment of considered dimension. Precisely, instead of considering the cell mass at a distance r on the sphere, we consider a cell mass at a distance r on a circle, which results in a linear factor instead of a quadratic one in the proper formula, that now reads as \begin{equation}\label{initial_condition_2D}
p(r,0) = 2 \pi r \Big( 1 - \Big(\frac{r}{\tilde{\sigma_i}}\Big)^{\tilde{q}} \Big)\mathds{1}_{\left[0,\tilde{\sigma_{i}}\right]}(r).
\end{equation}
Analogously to 3D case now instead of \eqref{eq:defL_2d} we consider
\begin{equation}\label{eq:defL_2d_bis}
L(R,r) =  \frac{\alpha}{\pi^2 \sigma^2} \, \left[\frac{\pi}{2} - \arcsin \max \left(\frac{R^2+r^2-\sigma^2}{2Rr}, -1\right)  \right] \, \mathds{1}_{|R-r|\leq \sigma}.
\end{equation}
\fi

\noindent Finally, the initial condition for the 2D case is similar to \eqref{initial_condition} with the only adjustment of considered dimension i.e. term $2\pi r$ instead of $4 \pi r^2$. Analogously to 3D case now instead of \eqref{eq:defL_2d} we consider
\begin{equation}\label{eq:defL_2d_bis}
L(R,r) =  \frac{\alpha}{\pi^2 \sigma^2} \, \left[\frac{\pi}{2} - \arcsin \max \left(\frac{R^2+r^2-\sigma^2}{2Rr}, -1\right)  \right] \, \mathds{1}_{|R-r|\leq \sigma}.
\end{equation}

\noindent Now, \tabref{q2D} presents the obtained relative error and the order of convergence of the numerical scheme applied to solve \eqref{non-local_proliferation_polar_2D}. Similarly to 3D case, the presented simulations were conducted on the spatial domain $R_0$ equal to 2 and smallest $\Delta t = 1.5625 \cdot 10^{-5}$, which gives us $1.28\cdot 10^{5}$ mass points for the smallest discretisation, whereas the simulation time $T$ was equal to 10. 
\begin{table}[ht]
\begin{tabular}{r|c|c}
\hline
$\Delta t = \Delta r$ & $\textrm{Err}\,(\Delta t)$ &$q_{\Delta t}$
\\\hline \hline
1.5625$\cdot 10^{-5}$ & 7.458465570038538 $\cdot 10^{3}$ & \bf{--} 
\\\hline
3.125$\cdot 10^{-5}$ & 1.0997423574488038$\cdot 10^{-2}$ & 0.5602148149156739
\\\hline
6.25$\cdot 10^{-5}$ & 1.643673194250337$\cdot 10^{-2}$ &  0.5797579061447363
\\\hline
1.25$\cdot 10^{-4}$ & 2.510571793647852$\cdot 10^{-2}$ &  0.6110925001107623
\\\hline
2.5$\cdot 10^{-4}$ & 3.908330914430527$\cdot 10^{-2}$ & 0.6385366423383074
\\\hline
5.0$\cdot 10^{-4}$ & 6.313477587243083$\cdot 10^{-2}$ & 0.691882264852459
\\\hline
1.0$\cdot 10^{-3}$ & 1.0560403857421896$\cdot 10^{-1}$ & 0.7421582142026246
\\\hline
\end{tabular}
\caption{
Error computed in metric $\textrm{Err}\,(\Delta t)$ \eqref{blad_zbieznosci_2D} together with corresponding order of convergence \eqref{rzad_zbieznosci} for the model given by \eqref{non-local_proliferation_polar_2D}.}\label{q2D}
\end{table}

%%%%%%%%%%%%%%%%%%%%%%%%%%%%%%%%%%%%%%%%%%%%%%%%%%%%%%%%%%%%%%%%%%%%%%%%%%%
%%%%%%%%%%%%%%%%%%%%%%%%%%%%%%%%%%%%%%%%%%%%%%%%%%%%%%%%%%%%%%%%%%%%%%%%%%%
%%%%%%%%%%%%%%%%%%%%%%%%%%%%%%%%%%%%%%%%%%%%%%%%%%%%%%%%%%%%%%%%%%%%%%%%%%%

\section*{Acknowledgments} B.~Miasojedow, J.~Skrzeczkowski, and Z.~Szyma\'nska  acknowledge the support from the National Science Centre, Poland -- grant No. 2017/26/M/ST1/00783. %P.~Gwiazda was supported by the National Science Centre, Poland -- grant No. 2019/35/N/ST1/03459. 
The calculations were made with the support of the Interdisciplinary Centre for Mathematical and Computational Modelling of the University of Warsaw under the computational grant no. G79-28. \\
\indent All authors would like to express their gratitude to Michał Dzikowski from the Interdisciplinary Centre for Mathematical and Computational Modelling of the University of Warsaw for his valuable help in performing scientific computing. 
\bibliographystyle{abbrv}
\bibliography{bibliography}

\begin{thebibliography}{10}

\bibitem{MR4045015}
A.~S. Ackleh, N.~Saintier, and J.~Skrzeczkowski.
\newblock Sensitivity equations for measure-valued solutions to transport
  equations.
\newblock {\em Math. Biosci. Eng.}, 17(1):514--537, 2020.

\bibitem{MR3138105}
{\AA}.~Br{\"a}nnstr{\"o}m, L.~Carlsson, and D.~Simpson.
\newblock On the convergence of the escalator boxcar train.
\newblock {\em SIAM J. Numer. Anal.}, 51(6):3213--3231, 2013.

\bibitem{byrne1995}
H.~Byrne and M.~Chaplain.
\newblock Growth of non-necrotic tumours in the presence and absence of
  inhibitors.
\newblock {\em Math. Biosci.}, 130:151--181, 1995.

\bibitem{byrne1996}
H.~Byrne and M.~Chaplain.
\newblock Growth of necrotic tumours in the presence and absence of inhibitors.
\newblock {\em Math. Biosci.}, 135:187--216, 1996.

\bibitem{byrne1998}
H.~Byrne and M.~Chaplain.
\newblock Necrosis and apoptosis: distinct cell loss mechanisms in a
  mathematical model of avascular tumour growth.
\newblock {\em Comput Math Methods Med.}, 1:223--235, 1998.

\bibitem{carrillo2012}
J.~A. Carrillo, R.~M. Colombo, P.~Gwiazda, and A.~Ulikowska.
\newblock Structured populations, cell growth and measure valued balance laws.
\newblock {\em J. Differential Equations}, 252(4):3245--3277, 2012.

\bibitem{MR3870087}
J.~A. Carrillo, S.~Fagioli, F.~Santambrogio, and M.~Schmidtchen.
\newblock Splitting schemes and segregation in reaction cross-diffusion
  systems.
\newblock {\em SIAM J. Math. Anal.}, 50(5):5695--5718, 2018.

\bibitem{MR3986559}
J.~A. Carrillo, P.~Gwiazda, K.~Kropielnicka, and A.~K. Marciniak-Czochra.
\newblock The escalator boxcar train method for a system of age-structured
  equations in the space of measures.
\newblock {\em SIAM J. Numer. Anal.}, 57(4):1842--1874, 2019.

\bibitem{carrillo2014}
J.~A. Carrillo, P.~Gwiazda, and A.~Ulikowska.
\newblock Splitting-particle methods for structured population models:
  convergence and applications.
\newblock {\em Math. Models Methods Appl. Sci.}, 24(11):2171--2197, 2014.

\bibitem{chen2005}
X.~Chen, S.~Cui, and A.~Friedman.
\newblock A hyperbolic free boundary problem modeling tumor growth: asymptotic
  behavior.
\newblock {\em Trans. Amer. Math. Soc.}, 357(12):4771--4804, 2005.

\bibitem{MR2888301}
A.~Chertock, J.-G. Liu, and T.~Pendleton.
\newblock Convergence of a particle method and global weak solutions of a
  family of evolutionary {PDE}s.
\newblock {\em SIAM J. Numer. Anal.}, 50(1):1--21, 2012.

\bibitem{MR2050900}
R.~M. Colombo and A.~Corli.
\newblock A semilinear structure on semigroups in a metric space.
\newblock {\em Semigroup Forum}, 68(3):419--444, 2004.

\bibitem{MR731212}
G.-H. Cottet and P.-A. Raviart.
\newblock Particle methods for the one-dimensional {V}lasov-{P}oisson
  equations.
\newblock {\em SIAM J. Numer. Anal.}, 21(1):52--76, 1984.

\bibitem{cui2003}
S.~Cui and A.~Friedman.
\newblock A hyperbolic free boundary problem modeling tumor growth.
\newblock {\em Interfaces Free. Boundaries.}, 5(2):159--181, 2003.

\bibitem{deRoos1988}
A.~M. de~Roos.
\newblock Numerical methods for structured population models: the escalator
  boxcar train.
\newblock {\em Numer. Methods Partial Differential Equations}, 4(3):173--195,
  1988.

\bibitem{de1997gentle}
A.~M. de~Roos.
\newblock A gentle introduction to physiologically structured population
  models.
\newblock In {\em Structured-population models in marine, terrestrial, and
  freshwater systems}, pages 119--204. Springer, 1997.

\bibitem{de2001physiologically}
A.~M. De~Roos and L.~Persson.
\newblock Physiologically structured models--from versatile technique to
  ecological theory.
\newblock {\em Oikos}, 94(1):51--71, 2001.

\bibitem{MR3529992}
Y.~Duan and J.-G. Liu.
\newblock Error estimate of the particle method for the {$b$}-equation.
\newblock {\em Methods Appl. Anal.}, 23(2):119--154, 2016.

\bibitem{our_book_ACPJ}
C.~D\"ull, P.~Gwiazda, A.~Marciniak-Czochra, and J.~Skrzeczkowski.
\newblock {\em Spaces of Measures and their Applications to Structured
  Population Models}.
\newblock Cambridge Monographs on Applied and Computational Mathematics.
  Cambridge University Press, 2021.

\bibitem{MR2597943}
L.~C. Evans.
\newblock {\em Partial differential equations}, volume~19 of {\em Graduate
  Studies in Mathematics}.
\newblock American Mathematical Society, Providence, RI, second edition, 2010.

\bibitem{MR3342408}
J.~H.~M. Evers, S.~C. Hille, and A.~Muntean.
\newblock Mild solutions to a measure-valued mass evolution problem with flux
  boundary conditions.
\newblock {\em J. Differential Equations}, 259(3):1068--1097, 2015.

\bibitem{MR3507552}
J.~H.~M. Evers, S.~C. Hille, and A.~Muntean.
\newblock Measure-valued mass evolution problems with flux boundary conditions
  and solution-dependent velocities.
\newblock {\em SIAM J. Math. Anal.}, 48(3):1929--1953, 2016.

\bibitem{FalsterE2719}
D.~S. Falster, {\r A}.~Br{\"a}nnstr{\"o}m, M.~Westoby, and U.~Dieckmann.
\newblock Multitrait successional forest dynamics enable diverse competitive
  coexistence.
\newblock {\em Proc. Natl. Acad. Sci. U.S.A.}, 114(13):E2719--E2728, 2017.

\bibitem{Falster2016}
D.~S. Falster, R.~G. FitzJohn, {\r A}.~Br{\"a}nnstr{\"o}m, U.~Dieckmann, and
  M.~Westoby.
\newblock Plant: a package for modelling forest trait ecology and evolution.
\newblock {\em Methods Ecol. Evol.}, 7(2):136--146, 2016.

\bibitem{Folland.1984}
G.~B. Folland.
\newblock {\em Real analysis}.
\newblock Pure and Applied Mathematics (New York). John Wiley \& Sons, Inc.,
  New York, 1984.

\bibitem{friedman2006b}
A.~Friedman and B.~Hu.
\newblock Asymptotic stability for a free boundary problem arising in a tumor
  model.
\newblock {\em J Differ Equ.}, 227(2):598--639, 2006.

\bibitem{friedman2006a}
A.~Friedman and B.~Hu.
\newblock Bifurcation from stability to instability for a free boundary problem
  arising in a tumor model.
\newblock {\em Arch Ration Mech Anal.}, 180(2):293--330, 2006.

\bibitem{MR987390}
K.~Ganguly and H.~D. Victory, Jr.
\newblock On the convergence of particle methods for multidimensional
  {V}lasov-{P}oisson systems.
\newblock {\em SIAM J. Numer. Anal.}, 26(2):249--288, 1989.

\bibitem{MR3632260}
Y.~Gao and J.-G. Liu.
\newblock Global convergence of a sticky particle method for the modified
  {C}amassa-{H}olm equation.
\newblock {\em SIAM J. Math. Anal.}, 49(2):1267--1294, 2017.

\bibitem{MR1040146}
J.~Goodman, T.~Y. Hou, and J.~Lowengrub.
\newblock Convergence of the point vortex method for the {$2$}-{D} {E}uler
  equations.
\newblock {\em Comm. Pure Appl. Math.}, 43(3):415--430, 1990.

\bibitem{MR4027078}
P.~Gwiazda, S.~C. Hille, K.~{\L}yczek, and A.~{\'{S}}wierczewska-Gwiazda.
\newblock Differentiability in perturbation parameter of measure solutions to
  perturbed transport equation.
\newblock {\em Kinet. Relat. Models}, 12(5):1093--1108, 2019.

\bibitem{MR3267354}
P.~Gwiazda, J.~Jab{\l}o\'{n}ski, A.~Marciniak-Czochra, and A.~Ulikowska.
\newblock Analysis of particle methods for structured population models with
  nonlocal boundary term in the framework of bounded {L}ipschitz distance.
\newblock {\em Numer. Methods Partial Differential Equations},
  30(6):1797--1820, 2014.

\bibitem{MR3461738}
P.~Gwiazda, K.~Kropielnicka, and A.~Marciniak-Czochra.
\newblock The escalator boxcar train method for a system of age-structured
  equations.
\newblock {\em Netw. Heterog. Media}, 11(1):123--143, 2016.

\bibitem{MR2644146}
P.~Gwiazda, T.~Lorenz, and A.~Marciniak-Czochra.
\newblock A nonlinear structured population model: {L}ipschitz continuity of
  measure-valued solutions with respect to model ingredients.
\newblock {\em J. Differential Equations}, 248(11):2703--2735, 2010.

\bibitem{MR2746205}
P.~Gwiazda and A.~Marciniak-Czochra.
\newblock Structured population equations in metric spaces.
\newblock {\em J. Hyperbolic Differ. Equ.}, 7(4):733--773, 2010.

\bibitem{szymanska_github}
P.~Gwiazda, B.~Miasojedow, J.~Skrzeczkowski, and Z.~Szyma\'nska.
\newblock Non-local-proliferation-model.
\newblock
  \url{https://github.com/Zuzanna-Szymanska/Non-local-proliferation-model},
  2021.

\bibitem{MR3591128}
P.~Gwiazda, P.~Orli\'{n}ski, and A.~Ulikowska.
\newblock Finite range method of approximation for balance laws in measure
  spaces.
\newblock {\em Kinet. Relat. Models}, 10(3):669--688, 2017.

\bibitem{jablonski2013efficient}
J.~Jablonski and A.~Marciniak-Czochra.
\newblock Efficient algorithms computing distances between radon measures on r.
\newblock {\em arXiv preprint arXiv:1304.3501}, 2013.

\bibitem{MR802214}
P.-A. Raviart.
\newblock An analysis of particle methods.
\newblock In {\em Numerical methods in fluid dynamics ({C}omo, 1983)}, volume
  1127 of {\em Lecture Notes in Math.}, pages 243--324. Springer, Berlin, 1985.

\bibitem{MR4066016}
J.~Skrzeczkowski.
\newblock Measure solutions to perturbed structured population
  models---differentiability with respect to perturbation parameter.
\newblock {\em J. Differential Equations}, 268(8):4119--4182, 2020.

\bibitem{szymanska2021}
Z.~Szyma\'nska, B.~Miasojedow, J.~Skrzeczkowski, and P.Gwiazda.
\newblock Bayesian inference of a non-local proliferation model.
\newblock {\em arXiv preprint arXiv:2106.05955}, pages 1--29, 2021.

\bibitem{MR2997595}
A.~Ulikowska.
\newblock An age-structured two-sex model in the space of {R}adon measures:
  well posedness.
\newblock {\em Kinet. Relat. Models}, 5(4):873--900, 2012.

\bibitem{MR2647756}
M.~Westdickenberg and J.~Wilkening.
\newblock Variational particle schemes for the porous medium equation and for
  the system of isentropic {E}uler equations.
\newblock {\em M2AN Math. Model. Numer. Anal.}, 44(1):133--166, 2010.

\bibitem{zhang2017performance}
L.~Zhang, U.~Dieckmann, and {\AA}.~Br{\"a}nnstr{\"o}m.
\newblock On the performance of four methods for the numerical solution of
  ecologically realistic size-structured population models.
\newblock {\em Methods Ecol. Evol.}, 8(8):948--956, 2017.

\end{thebibliography}
\end{document}